\documentclass[12pt]{report}
\usepackage{titlesec}
\usepackage{amsmath}
\usepackage{amsthm}
\usepackage{ amssymb}
\usepackage{mathtools}
\usepackage{graphicx}
\usepackage{comment}
\usepackage[dvipsnames]{xcolor}
\usepackage{tikz}
\usetikzlibrary{calc, shapes, arrows.meta, positioning, backgrounds, decorations.pathmorphing}
\usepackage{float}
\usepackage[framemethod=tikz]{mdframed}  
\usepackage{pgfplots, pgfplotstable}
\usepackage{indentfirst}
\pgfplotsset{compat=1.17}
\usepackage[a4paper,width=150mm,top=25mm,bottom=25mm]{geometry}
\usepackage{fancyhdr}
\graphicspath{ {images/} }
\usepackage{ifthen}
\usepackage{hyperref}
\usepackage{caption}  

\titleformat{\chapter}[display]   
{\normalfont\huge\bfseries}{\chaptertitlename\ \thechapter}{20pt}{\Huge}   
\titlespacing*{\chapter}{0pt}{-30pt}{40pt}

\theoremstyle{plain}
\newtheorem{theorem}{Theorem}[section] 
\newmdtheoremenv{bigtheo}{Theorem}
\newtheorem{corollary}{Corollary}[section]
\newtheorem{lemma}[theorem]{Lemma}
\newtheorem{prop}[section]{Proposition}
\newtheorem*{remark}{Remark}
\newtheorem{definition}{Definition}[section] 

\newcommand{\N}{\ensuremath{\mathbb{N}}}
\newcommand{\R}{\ensuremath{\mathbb{R}}}
\newcommand{\Z}{\ensuremath{\mathbb{Z}}}
\newcommand{\Q}{\ensuremath{\mathbb{Q}}}
\newcommand{\PP}{\ensuremath{\mathbb{P}}}

\newcommand{\indep}{\perp \!\!\! \perp}

\title{Part III Essay}

\begin{document}

\title{

{\includegraphics[width=0.3\textwidth]{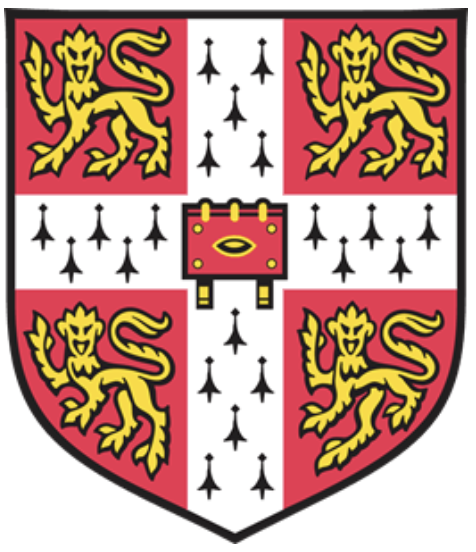}}\\
{The KPZ Fixed Point and the Directed Landscape}\\
{Part III Essay}\\
{Pantelis Tassopoulos}
}

\date{May 2024}

\maketitle

\tableofcontents

\chapter*{Overview}

\indent The term \textquotedblleft KPZ\textquotedblright stands for the initials of three physicists, namely Kardar, Parisi and Zhang, which, in 1986 conjectured the existence of universal scaling behaviours for many random growth processes in the plane \cite{kardar1986dynamic}.\\

A process is said to belong to the KPZ universality class if one can associate to it an appropriate \textquotedblleft height function\textquotedblright and show that its 3:2:1 (time : space : fluctuation) scaling limit, see \ref{eq: KPZ scaling}, converges to a universal random process, the \textbf{KPZ fixed point}. Alternatively membership is loosely characterised by having: 1. Local dynamics; 2.
A smoothing mechanism; 3. Slope dependent growth rate (lateral growth); 4. Space-time random forcing with rapid decay of correlations \cite{matetski2021kpz}, see section \ref{sec: non-gaussian}.\\

The central object that we will study is the so-called KPZ fixed point, which belongs to the KPZ universality class. Many strides have been made in the last couple of decades in this field, with constructions of the KPZ fixed point from certain processes such as the totally asymmetric simple exclusion process (with arbitrary initial condition) \cite{matetski2021kpz} and Brownian last passage percolation \cite{nica2020one}.\\

Crucial in the study of the KPZ fixed point, is the Airy Line Ensemble, a collection of random non-intersecting continuous curves satisfying a certain Brownian resampling property, namely, the \textit{Brownian Gibbs Property}. Its locally Brownian nature, to be made precise later and the properties of the geometry of last passage percolation over it, led to the construction of the Directed Landscape and Airy Sheet, in the breakthrough paper of Dauvergne, Ortmann and Virag (DOV) \cite{DOV}. These objects lend themselves to description by a variational formula, which lies at the crux of the result of by Sarkar and Virag \cite{sarkar2021brownian}, showing the absolute continuity of the KPZ fixed point on compacts with respect to Brownian motion. \\

In this essay, we will: 1. delineate the origins of KPZ universality; 2. describe and motivate canonical models; 3. give an overview of recent developments, especially those in the DOV paper; 4. present the strategy of and key points in the proof of the aforementioned absolute continuity result of \cite{sarkar2021brownian}; 5. conclude with remarks for future directions. The presentation is such that the content is displayed in an as much self-contained manner as possible, aimed at a motivated audience having mastered the fundamentals of the theory of probability.

\subsection*{Acknowledgements}
\noindent I would like to express my gratitude to Dr. Sourav Sarkar for not only proposing this topic but also for his guidance and constructive feedback, making all this possible. 

\chapter{Introduction}

\section{Gaussian Universality}

The central theme underlying this essay is that of universality. Universality in random systems has been a driving factor in research guiding developments in Probability and Mathematical Physics. The reader will be familiar with the classical example of the Gaussian universality class with its most well-known manifestation being the central limit theorem. There are more subtle manifestations of this universality class in the behaviour of the interface of random growth process, in particular, in the random deposition model which we briefly describe next.\\

In the \textbf{random deposition model}, blocks of unit side length are sequentially deposited independently and at random into columns of unit length where the vertices of the blocks take values in \(\Z\) and the separation in time of consecutive blocks falling in each column is distributed according to an exponential distribution of rate \(\lambda >0\). Allowing this model evolve over time leads to a picture as in the left diagram in figure \ref{Random deposition}. Note due to the memorylessness property of the exponential distribution, the future evolution of the system is determined by its present state. Also, independence of the location of block depositions implies independence of the growth of individual columns. One can define the height function \(h:\R_{+}\times\R \mapsto \R\), where \(h(t,x)\) records the number of blocks that have landed above site \(x\), before and including time \(t\).\\

We now analyse the asymptotic scaling properties of the height function in time. For simplicity, set \(\lambda = 1\). Define the random waiting times \(T_{x,i}\) as the time that elapses for the \(i\)-th block to land in column \(x\). We have the equivalent characterisations
\[
h(t,x)<n \quad \iff \quad \displaystyle \Sigma_{i=1}^nT_{x,i}>t\,.
\]

\noindent The times \(T_{x,i}\) are i.i.d. and more generally one observes that
\[
h(t,x) = \sup \{k\in \N:  \displaystyle \Sigma_{i=1}^kT_{x,i}\leq t\}\,.
\]
\noindent Furthermore, the law of large numbers and the central limit theorem imply

\[
\displaystyle \lim_{t\to\infty}\frac{h(t,x)}{t} = 1, \quad \text{and} \quad \displaystyle \lim_{t\to\infty}\frac{h(t,x)-t}{t^{1/2}} \stackrel{d}{\rightarrow} \mathcal{N}(0,1)\,,
\]

\noindent where the second convergence is in distribution, to a standard normal random variable, for each site \(x\in \R\) \cite{corwin2016kardar}. For a demonstration of the former, observe that 
\[
\displaystyle\Sigma_{i=1}^{h(t,x)}T_{x,i}\leq t \quad \text{and} \quad \displaystyle \Sigma_{i=1}^{h(t,x)+1}T_{x,i}> t
\]
and so \[
\displaystyle\frac{\Sigma_{i=1}^{h(t,x)}T_{x,i}}{h(t,x)}\leq \frac{t}{h(t,x)} <\displaystyle \displaystyle\frac{\Sigma_{i=1}^{h(t,x)+1}T_{x,i}}{h(t,x)}.\]
By independence of the arrival times and their common distribution, an application of the Borel-Cantelli lemma yields that for all \(x\in\R\) almost surely, as \(t\to\infty\), \(h(t,x)\to \infty\). Thus, noting that \((h(t,x)+1)/h(t,x)\to 1, t \to \infty\) and applying the strong law of large numbers, gives the conclusion \cite{mitov2014renewal}. Similar arguments yield the convergence in law. Since fluctuations of \(h(t,x)\) about its mean are of order \(t^{1/2}\), the above model is said to be in the Gaussian universality class. Notice that the limiting scaling behaviour of \(h(t,x)\) does not depend on the particular distribution of the arrival times. This however, changes drastically when one alters the growth rules of the model.\\

\vspace{0.1in}
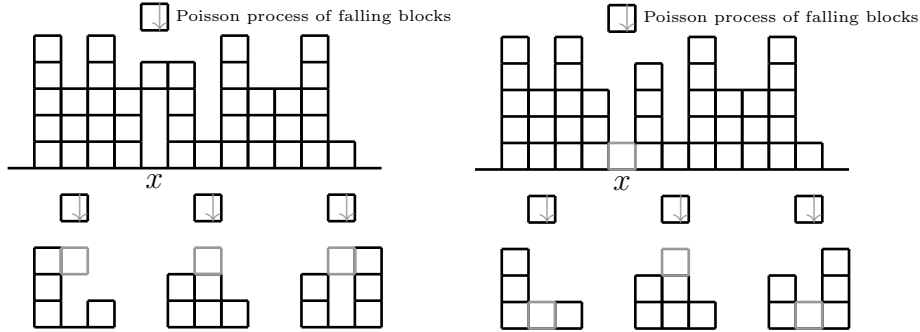
\begin{figure}[H]
    \centering
\begin{picture}(120,100)(-75,-55)
\linethickness{1pt}
\put(0,0){\line(1,0){140}}

\put(10,0){\line(0,1){10}}
\put(20,0){\line(0,1){10}}
\put(10,10){\line(1,0){10}}

\put(10,10){\line(0,1){10}}
\put(20,10){\line(0,1){10}}
\put(10,20){\line(1,0){10}}

\put(10,20){\line(0,1){10}}
\put(20,20){\line(0,1){10}}
\put(10,30){\line(1,0){10}}

\put(10,30){\line(0,1){10}}
\put(20,30){\line(0,1){10}}
\put(10,40){\line(1,0){10}}

\put(10,40){\line(0,1){10}}
\put(20,40){\line(0,1){10}}
\put(10,50){\line(1,0){10}}

\put(20,0){\line(0,1){10}}
\put(30,0){\line(0,1){10}}
\put(20,10){\line(1,0){10}}

\put(20,10){\line(0,1){10}}
\put(30,10){\line(0,1){10}}
\put(20,20){\line(1,0){10}}

\put(20,20){\line(0,1){10}}
\put(30,20){\line(0,1){10}}
\put(20,30){\line(1,0){10}}

\put(30,0){\line(0,1){10}}
\put(40,0){\line(0,1){10}}
\put(30,10){\line(1,0){10}}

\put(30,10){\line(0,1){10}}
\put(40,10){\line(0,1){10}}
\put(30,20){\line(1,0){10}}

\put(30,20){\line(0,1){10}}
\put(40,20){\line(0,1){10}}
\put(30,30){\line(1,0){10}}

\put(30,30){\line(0,1){10}}
\put(40,30){\line(0,1){10}}
\put(30,40){\line(1,0){10}}

\put(30,40){\line(0,1){10}}
\put(40,40){\line(0,1){10}}
\put(30,50){\line(1,0){10}}

\put(40,0){\line(0,1){10}}
\put(50,0){\line(0,1){10}}
\put(40,10){\line(1,0){10}}

\put(40,10){\line(0,1){10}}
\put(50,10){\line(0,1){10}}
\put(40,20){\line(1,0){10}}

\put(40,20){\line(0,1){10}}
\put(50,20){\line(0,1){10}}
\put(40,30){\line(1,0){10}}

\put(50,52){\line(0,1){10}}
\put(60,52){\line(0,1){10}}
\put(50,62){\line(1,0){10}}
\put(50,52){\line(1,0){10}}
\color{black!40}
\put(54,57){\makebox(0,0)[l]{$\downarrow$}}
\color{black}
\put(63,57){\makebox(0,0)[l]{\tiny{Poisson process of falling blocks}}}
\color{black!40}
\put(50,0){\line(0,1){10}}
\put(60,0){\line(0,1){10}}
\put(50,10){\line(1,0){10}}
\put(50,0){\line(1,0){10}}
\color{black!40}

\put(60,0){\line(0,1){10}}
\color{black}
\put(70,0){\line(0,1){10}}
\put(60,10){\line(1,0){10}}

\put(60,10){\line(0,1){10}}
\put(70,10){\line(0,1){10}}
\put(60,20){\line(1,0){10}}

\put(60,20){\line(0,1){10}}
\put(70,20){\line(0,1){10}}
\put(60,30){\line(1,0){10}}

\put(60,30){\line(0,1){10}}
\put(70,30){\line(0,1){10}}
\put(60,40){\line(1,0){10}}

\put(70,0){\line(0,1){10}}
\put(80,0){\line(0,1){10}}
\put(70,10){\line(1,0){10}}

\put(80,0){\line(0,1){10}}
\put(90,0){\line(0,1){10}}
\put(80,10){\line(1,0){10}}

\put(80,10){\line(0,1){10}}
\put(90,10){\line(0,1){10}}
\put(80,20){\line(1,0){10}}

\put(80,20){\line(0,1){10}}
\put(90,20){\line(0,1){10}}
\put(80,30){\line(1,0){10}}

\put(80,30){\line(0,1){10}}
\put(90,30){\line(0,1){10}}
\put(80,40){\line(1,0){10}}

\put(80,40){\line(0,1){10}}
\put(90,40){\line(0,1){10}}
\put(80,50){\line(1,0){10}}

\put(90,0){\line(0,1){10}}
\put(100,0){\line(0,1){10}}
\put(90,10){\line(1,0){10}}

\put(90,10){\line(0,1){10}}
\put(100,10){\line(0,1){10}}
\put(90,20){\line(1,0){10}}

\put(90,20){\line(0,1){10}}
\put(100,20){\line(0,1){10}}
\put(90,30){\line(1,0){10}}

\put(100,0){\line(0,1){10}}
\put(110,0){\line(0,1){10}}
\put(100,10){\line(1,0){10}}

\put(100,10){\line(0,1){10}}
\put(110,10){\line(0,1){10}}
\put(100,20){\line(1,0){10}}

\put(100,20){\line(0,1){10}}
\put(110,20){\line(0,1){10}}
\put(100,30){\line(1,0){10}}

\put(110,0){\line(0,1){10}}
\put(120,0){\line(0,1){10}}
\put(110,10){\line(1,0){10}}

\put(110,10){\line(0,1){10}}
\put(120,10){\line(0,1){10}}
\put(110,20){\line(1,0){10}}

\put(110,20){\line(0,1){10}}
\put(120,20){\line(0,1){10}}
\put(110,30){\line(1,0){10}}

\put(110,30){\line(0,1){10}}
\put(120,30){\line(0,1){10}}
\put(110,40){\line(1,0){10}}

\put(110,40){\line(0,1){10}}
\put(120,40){\line(0,1){10}}
\put(110,50){\line(1,0){10}}

\put(120,0){\line(0,1){10}}
\put(130,0){\line(0,1){10}}
\put(120,10){\line(1,0){10}}
\put(55,-5){\makebox(0,0){$x$}}
%


\put(10,-60){\line(1,0){30}}

\put(10,-60){\line(0,1){10}}
\put(20,-60){\line(0,1){10}}
\put(10,-50){\line(1,0){10}}

\put(10,-50){\line(0,1){10}}
\put(20,-50){\line(0,1){10}}
\put(10,-40){\line(1,0){10}}

\put(10,-40){\line(0,1){10}}
\put(20,-40){\line(0,1){10}}
\put(10,-30){\line(1,0){10}}

\put(30,-60){\line(0,1){10}}
\put(40,-60){\line(0,1){10}}
\put(30,-50){\line(1,0){10}}

\put(20,-20){\line(0,1){10}}
\put(30,-20){\line(0,1){10}}
\put(20,-10){\line(1,0){10}}
\put(20,-20){\line(1,0){10}}
\color{black!40}
\put(24,-15){\makebox(0,0)[l]{$\downarrow$}}

\put(20,-60){\line(0,1){10}}
\put(30,-60){\line(0,1){10}}
\put(20,-50){\line(1,0){10}}
\put(20,-60){\line(1,0){10}}

\color{black}


\put(60,-60){\line(1,0){30}}

\put(60,-60){\line(0,1){10}}
\put(70,-60){\line(0,1){10}}
\put(60,-50){\line(1,0){10}}

\put(60,-50){\line(0,1){10}}
\put(70,-50){\line(0,1){10}}
\put(60,-40){\line(1,0){10}}

\put(70,-60){\line(0,1){10}}
\put(80,-60){\line(0,1){10}}
\put(70,-50){\line(1,0){10}}

\put(70,-50){\line(0,1){10}}
\put(80,-50){\line(0,1){10}}
\put(70,-40){\line(1,0){10}}

\put(80,-60){\line(0,1){10}}
\put(90,-60){\line(0,1){10}}
\put(80,-50){\line(1,0){10}}

\put(70,-20){\line(0,1){10}}
\put(80,-20){\line(0,1){10}}
\put(70,-10){\line(1,0){10}}
\put(70,-20){\line(1,0){10}}
\color{black!40}
\put(74,-15){\makebox(0,0)[l]{$\downarrow$}}

\put(70,-40){\line(0,1){10}}
\put(80,-40){\line(0,1){10}}
\put(70,-30){\line(1,0){10}}
\put(70,-40){\line(1,0){10}}

\color{black}


\put(110,-60){\line(1,0){30}}

\put(110,-60){\line(0,1){10}}
\put(120,-60){\line(0,1){10}}
\put(110,-50){\line(1,0){10}}

\put(110,-50){\line(0,1){10}}
\put(120,-50){\line(0,1){10}}
\put(110,-40){\line(1,0){10}}

\put(130,-60){\line(0,1){10}}
\put(140,-60){\line(0,1){10}}
\put(130,-50){\line(1,0){10}}

\put(130,-50){\line(0,1){10}}
\put(140,-50){\line(0,1){10}}
\put(130,-40){\line(1,0){10}}

\put(130,-40){\line(0,1){10}}
\put(140,-40){\line(0,1){10}}
\put(130,-30){\line(1,0){10}}

\put(120,-20){\line(0,1){10}}
\put(130,-20){\line(0,1){10}}
\put(120,-10){\line(1,0){10}}
\put(120,-20){\line(1,0){10}}
\color{black!40}
\put(124,-15){\makebox(0,0)[l]{$\downarrow$}}

\put(120,-60){\line(0,1){10}}
\put(130,-60){\line(0,1){10}}
\put(120,-50){\line(1,0){10}}
\put(120,-60){\line(1,0){10}}

\end{picture}\label{Random deposition}

    \caption{Left: Ballistic deposition model. Growth occurs when blocks stick to first point of contact (denoted in \color{black!40}gray\color{black}). Right: Random deposition model, growth occurs when a falling block hits a pre-existing block from above (denoted in \color{black!40}gray\color{black})\protect\footnotemark.}
    \label{fig:enter-label}
    \begin{picture}(120,100)(100,-210)

\linethickness{1pt}
\put(0,0){\line(1,0){140}}

\put(10,0){\line(0,1){10}}
\put(20,0){\line(0,1){10}}
\put(10,10){\line(1,0){10}}

\put(10,10){\line(0,1){10}}
\put(20,10){\line(0,1){10}}
\put(10,20){\line(1,0){10}}

\put(10,20){\line(0,1){10}}
\put(20,20){\line(0,1){10}}
\put(10,30){\line(1,0){10}}

\put(10,30){\line(0,1){10}}
\put(20,30){\line(0,1){10}}
\put(10,40){\line(1,0){10}}

\put(10,40){\line(0,1){10}}
\put(20,40){\line(0,1){10}}
\put(10,50){\line(1,0){10}}

\put(20,0){\line(0,1){10}}
\put(30,0){\line(0,1){10}}
\put(20,10){\line(1,0){10}}

\put(20,10){\line(0,1){10}}
\put(30,10){\line(0,1){10}}
\put(20,20){\line(1,0){10}}

\put(20,20){\line(0,1){10}}
\put(30,20){\line(0,1){10}}
\put(20,30){\line(1,0){10}}

\put(30,0){\line(0,1){10}}
\put(40,0){\line(0,1){10}}
\put(30,10){\line(1,0){10}}

\put(30,10){\line(0,1){10}}
\put(40,10){\line(0,1){10}}
\put(30,20){\line(1,0){10}}

\put(30,20){\line(0,1){10}}
\put(40,20){\line(0,1){10}}
\put(30,30){\line(1,0){10}}

\put(30,30){\line(0,1){10}}
\put(40,30){\line(0,1){10}}
\put(30,40){\line(1,0){10}}

\put(30,40){\line(0,1){10}}
\put(40,40){\line(0,1){10}}
\put(30,50){\line(1,0){10}}

\put(40,0){\line(0,1){10}}
\put(50,0){\line(0,1){10}}
\put(40,10){\line(1,0){10}}

\put(40,10){\line(0,1){10}}
\put(50,10){\line(0,1){10}}
\put(40,20){\line(1,0){10}}

\put(40,20){\line(0,1){10}}
\put(50,20){\line(0,1){10}}
\put(40,30){\line(1,0){10}}

\put(50,52){\line(0,1){10}}
\put(60,52){\line(0,1){10}}
\put(50,62){\line(1,0){10}}
\put(50,52){\line(1,0){10}}

\color{black!40}
\put(54,57){\makebox(0,0)[l]{$\downarrow$}}
\color{black}

\put(63,57){\makebox(0,0)[l]{\tiny{Poisson process of falling blocks}}}
\put(50,30){\line(0,1){10}}
\put(60,30){\line(0,1){10}}
\put(50,40){\line(1,0){10}}
\put(50,30){\line(1,0){10}}
\put(60,0){\line(0,1){10}}
\put(70,0){\line(0,1){10}}
\put(60,10){\line(1,0){10}}

\put(60,10){\line(0,1){10}}
\put(70,10){\line(0,1){10}}
\put(60,20){\line(1,0){10}}

\put(60,20){\line(0,1){10}}
\put(70,20){\line(0,1){10}}
\put(60,30){\line(1,0){10}}

\put(60,30){\line(0,1){10}}
\put(70,30){\line(0,1){10}}
\put(60,40){\line(1,0){10}}

\put(70,0){\line(0,1){10}}
\put(80,0){\line(0,1){10}}
\put(70,10){\line(1,0){10}}

\put(80,0){\line(0,1){10}}
\put(90,0){\line(0,1){10}}
\put(80,10){\line(1,0){10}}

\put(80,10){\line(0,1){10}}
\put(90,10){\line(0,1){10}}
\put(80,20){\line(1,0){10}}

\put(80,20){\line(0,1){10}}
\put(90,20){\line(0,1){10}}
\put(80,30){\line(1,0){10}}

\put(80,30){\line(0,1){10}}
\put(90,30){\line(0,1){10}}
\put(80,40){\line(1,0){10}}

\put(80,40){\line(0,1){10}}
\put(90,40){\line(0,1){10}}
\put(80,50){\line(1,0){10}}

\put(90,0){\line(0,1){10}}
\put(100,0){\line(0,1){10}}
\put(90,10){\line(1,0){10}}

\put(90,10){\line(0,1){10}}
\put(100,10){\line(0,1){10}}
\put(90,20){\line(1,0){10}}

\put(90,20){\line(0,1){10}}
\put(100,20){\line(0,1){10}}
\put(90,30){\line(1,0){10}}

\put(100,0){\line(0,1){10}}
\put(110,0){\line(0,1){10}}
\put(100,10){\line(1,0){10}}

\put(100,10){\line(0,1){10}}
\put(110,10){\line(0,1){10}}
\put(100,20){\line(1,0){10}}

\put(100,20){\line(0,1){10}}
\put(110,20){\line(0,1){10}}
\put(100,30){\line(1,0){10}}

\put(110,0){\line(0,1){10}}
\put(120,0){\line(0,1){10}}
\put(110,10){\line(1,0){10}}

\put(110,10){\line(0,1){10}}
\put(120,10){\line(0,1){10}}
\put(110,20){\line(1,0){10}}

\put(110,20){\line(0,1){10}}
\put(120,20){\line(0,1){10}}
\put(110,30){\line(1,0){10}}

\put(110,30){\line(0,1){10}}
\put(120,30){\line(0,1){10}}
\put(110,40){\line(1,0){10}}

\put(110,40){\line(0,1){10}}
\put(120,40){\line(0,1){10}}
\put(110,50){\line(1,0){10}}

\put(120,0){\line(0,1){10}}
\put(130,0){\line(0,1){10}}
\put(120,10){\line(1,0){10}}
\put(55,-5){\makebox(0,0){$x$}}
%


\put(10,-60){\line(1,0){30}}

\put(10,-60){\line(0,1){10}}
\put(20,-60){\line(0,1){10}}
\put(10,-50){\line(1,0){10}}

\put(10,-50){\line(0,1){10}}
\put(20,-50){\line(0,1){10}}
\put(10,-40){\line(1,0){10}}

\put(10,-40){\line(0,1){10}}
\put(20,-40){\line(0,1){10}}
\put(10,-30){\line(1,0){10}}

\put(30,-60){\line(0,1){10}}
\put(40,-60){\line(0,1){10}}
\put(30,-50){\line(1,0){10}}

\put(20,-20){\line(0,1){10}}
\put(30,-20){\line(0,1){10}}
\put(20,-10){\line(1,0){10}}
\put(20,-20){\line(1,0){10}}
\color{black!40}
\put(24,-15){\makebox(0,0)[l]{$\downarrow$}}

\put(20,-40){\line(0,1){10}}
\put(30,-40){\line(0,1){10}}
\put(20,-30){\line(1,0){10}}
\put(20,-40){\line(1,0){10}}

\color{black}


\put(60,-60){\line(1,0){30}}

\put(60,-60){\line(0,1){10}}
\put(70,-60){\line(0,1){10}}
\put(60,-50){\line(1,0){10}}

\put(60,-50){\line(0,1){10}}
\put(70,-50){\line(0,1){10}}
\put(60,-40){\line(1,0){10}}

\put(70,-60){\line(0,1){10}}
\put(80,-60){\line(0,1){10}}
\put(70,-50){\line(1,0){10}}

\put(70,-50){\line(0,1){10}}
\put(80,-50){\line(0,1){10}}
\put(70,-40){\line(1,0){10}}

\put(80,-60){\line(0,1){10}}
\put(90,-60){\line(0,1){10}}
\put(80,-50){\line(1,0){10}}

\put(70,-20){\line(0,1){10}}
\put(80,-20){\line(0,1){10}}
\put(70,-10){\line(1,0){10}}
\put(70,-20){\line(1,0){10}}
\color{black!40}
\put(74,-15){\makebox(0,0)[l]{$\downarrow$}}

\put(70,-40){\line(0,1){10}}
\put(80,-40){\line(0,1){10}}
\put(70,-30){\line(1,0){10}}
\put(70,-40){\line(1,0){10}}

\color{black}


\put(110,-60){\line(1,0){30}}

\put(110,-60){\line(0,1){10}}
\put(120,-60){\line(0,1){10}}
\put(110,-50){\line(1,0){10}}

\put(110,-50){\line(0,1){10}}
\put(120,-50){\line(0,1){10}}
\put(110,-40){\line(1,0){10}}

\put(130,-60){\line(0,1){10}}
\put(140,-60){\line(0,1){10}}
\put(130,-50){\line(1,0){10}}

\put(130,-50){\line(0,1){10}}
\put(140,-50){\line(0,1){10}}
\put(130,-40){\line(1,0){10}}

\put(130,-40){\line(0,1){10}}
\put(140,-40){\line(0,1){10}}
\put(130,-30){\line(1,0){10}}

\put(120,-20){\line(0,1){10}}
\put(130,-20){\line(0,1){10}}
\put(120,-10){\line(1,0){10}}
\put(120,-20){\line(1,0){10}}
\color{black!40}
\put(124,-15){\makebox(0,0)[l]{$\downarrow$}}

\put(120,-40){\line(0,1){10}}
\put(130,-40){\line(0,1){10}}
\put(120,-30){\line(1,0){10}}
\put(120,-40){\line(1,0){10}}

\color{black}

\end{picture}\label{ballistic deposition}
\end{figure}
\footnotetext{Inspired by figure $0b$ in \cite{corwin2011kardarparisizhangreview}.}
\vspace{-4cm}

\section{Beyond Gaussian Universality}\label{sec: non-gaussian}

Introducing the modification that a falling block sticks to the side of the first block it becomes incident to, leads to the \textbf{ballistic deposition} model, as depicted in figure \ref{Random deposition}. Spatial correlation is expected, and indeed, in simulations (see figure $3$ in \cite{corwin2016kardar}), when compared over large epochs, the height function for random versus ballistic deposition have qualitative differences; the latter has a fractal nature with many overhangs created by the \textquotedblleft sticky\textquotedblright blocks. One also observes heuristically, from the simulations, that the fluctuations in the latter model are smaller than in random deposition, conjectured to grow at a rate of \(t^{1/3}\). Though the above conjecture and others regarding the exact scaling behaviour in large time scales remain unproven, the model exhibits the following three properties, expected to be manifested by any process that is in the KPZ universality class, namely, 

\begin{itemize}
    \item \textbf{Locality}: Height function at a column depends only on a neighbourhoud of the column
    \item \textbf{Smoothing}: Large valleys are quickly filled
    \item \textbf{Non-linear slope dependence}: Vertical growth rate depends non-linearly on the local slope of the height function
    \item \textbf{Space-time independent noise}: Growth is driven by noise that decorelates quickly in space/time and is not heavy tailed.
\end{itemize}

However, slightly changing the deposition geometry leads to an integrable probabilistic system, namely, the \textbf{corner growth model}, which was shown to be in the KPZ universality class\footnote{By \textquotedblleft  integrable \textquotedblright, one usually refers to the case where explicit formulas describing the dynamics of the system can be derived.}. The growth interface (\color{blue} blue \color{black} in figure \ref{fig:Corner growth}) is given by a piecewise-linear function where growth occurs at each local minimum (\(\lor\)) or local maximum (\(\land\)) after an exponentially distributed waiting time (independently from each other other local min/max). With initial data shaped like a \(45^{\circ}\) wedge, or more precisely, \(h(0,x) = |x|\), as in figure \ref{fig:Corner growth}.

\vspace{0.2in}

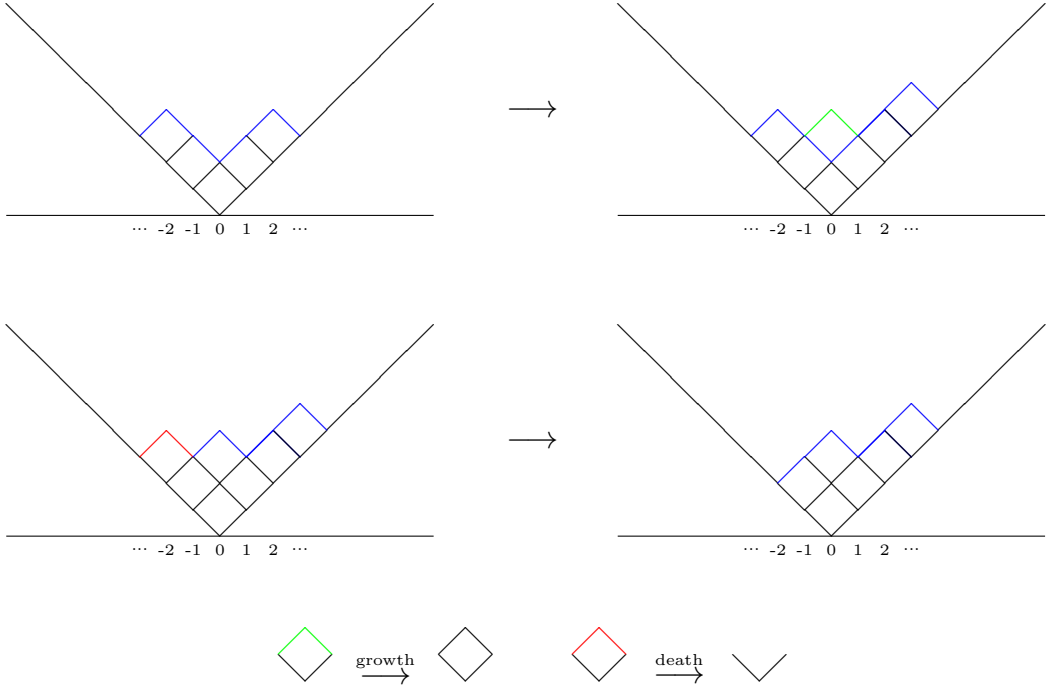
\begin{figure}[H]
\centering
\setlength{\unitlength}{1pt}
\begin{picture}(200,120)(-75, -60)\label{corner}
\put(0,0){\line(1,1){80}}
\put(0,0){\line(-1,1){80}}
\put(-10,10){\line(1,1){10}}
\put(10,10){\line(-1,1){10}}
\put(-20,20){\line(1,1){10}}
\color{blue}
\put(0,20){\line(-1,1){10}}
\put(-30,30){\line(1,1){10}}
\put(-10,30){\line(-1,1){10}}
\put(0,20){\line(1,1){10}}
\color{black}
\put(20,20){\line(-1,1){10}}

\color{blue}
\put(10,30){\line(1,1){10}}
\put(30,30){\line(-1,1){10}}
\color{black}

\put(-80,0){\line(1,0){160}}
\put(0,-5){\makebox(0,0){$\textrm{\tiny{0}}$}}
\put(-10,-5){\makebox(0,0){$\textrm{\tiny{-1}}$}}
\put(-20,-5){\makebox(0,0){$\textrm{\tiny{-2}}$}}
\put(-30,-5){\makebox(0,0){$\textrm{\tiny{...}}$}}
\put(10,-5){\makebox(0,0){$\textrm{\tiny{1}}$}}
\put(20,-5){\makebox(0,0){$\textrm{\tiny{2}}$}}
\put(30,-5){\makebox(0,0){$\textrm{\tiny{...}}$}}

\end{picture}
\begin{picture}(170,120)(-100, -60)\label{corner}
\put(0,0){\line(1,1){80}}
\put(0,0){\line(-1,1){80}}
\put(-10,10){\line(1,1){10}}
\put(10,10){\line(-1,1){10}}
\put(-20,20){\line(1,1){10}}
\color{blue}
\put(0,20){\line(-1,1){10}}
\put(-30,30){\line(1,1){10}}
\put(-10,30){\line(-1,1){10}}
\put(0,20){\line(1,1){10}}
\color{black}
\put(20,20){\line(-1,1){10}}

\color{blue}
\put(10,30){\line(1,1){10}}
\put(30,30){\line(-1,1){10}}
\color{black}

\color{green}
\put(-10,30){\line(1,1){10}}
\put(10,30){\line(-1,1){10}}
\color{black}

\put(30,30){\line(-1,1){10}}
\color{blue}
\put(10,30){\line(1,1){10}}
\put(20,40){\line(1,1){10}}
\put(40,40){\line(-1,1){10}}
\color{black}

\put(-80,0){\line(1,0){160}}
\put(0,-5){\makebox(0,0){$\textrm{\tiny{0}}$}}
\put(-10,-5){\makebox(0,0){$\textrm{\tiny{-1}}$}}
\put(-20,-5){\makebox(0,0){$\textrm{\tiny{-2}}$}}
\put(-30,-5){\makebox(0,0){$\textrm{\tiny{...}}$}}
\put(10,-5){\makebox(0,0){$\textrm{\tiny{1}}$}}
\put(20,-5){\makebox(0,0){$\textrm{\tiny{2}}$}}
\put(30,-5){\makebox(0,0){$\textrm{\tiny{...}}$}}

\end{picture}
\begin{picture}(200,120)(-75, -60)\label{corner}

\put(0,0){\line(1,1){80}}
\put(0,0){\line(-1,1){80}}
\put(-10,10){\line(1,1){10}}
\put(10,10){\line(-1,1){10}}
\put(-20,20){\line(1,1){10}}
\put(0,20){\line(-1,1){10}}
\color{red}
\put(-30,30){\line(1,1){10}}
\put(-10,30){\line(-1,1){10}}
\color{black}
\put(0,20){\line(1,1){10}}
\color{black}
\put(20,20){\line(-1,1){10}}

\color{blue}
\put(10,30){\line(1,1){10}}
\put(30,30){\line(-1,1){10}}
\put(-10,30){\line(1,1){10}}
\put(10,30){\line(-1,1){10}}
\color{black}

\put(30,30){\line(-1,1){10}}
\color{blue}
\put(10,30){\line(1,1){10}}
\put(20,40){\line(1,1){10}}
\put(40,40){\line(-1,1){10}}
\color{black}

\put(-80,0){\line(1,0){160}}
\put(0,-5){\makebox(0,0){$\textrm{\tiny{0}}$}}
\put(-10,-5){\makebox(0,0){$\textrm{\tiny{-1}}$}}
\put(-20,-5){\makebox(0,0){$\textrm{\tiny{-2}}$}}
\put(-30,-5){\makebox(0,0){$\textrm{\tiny{...}}$}}
\put(10,-5){\makebox(0,0){$\textrm{\tiny{1}}$}}
\put(20,-5){\makebox(0,0){$\textrm{\tiny{2}}$}}
\put(30,-5){\makebox(0,0){$\textrm{\tiny{...}}$}}

\end{picture}
\begin{picture}(170,120)(-100, -60)\label{corner}
\put(0,0){\line(1,1){80}}
\put(0,0){\line(-1,1){80}}
\put(-10,10){\line(1,1){10}}
\put(10,10){\line(-1,1){10}}
\color{blue}
\put(-20,20){\line(1,1){10}}
\color{black}
\put(0,20){\line(-1,1){10}}
\put(0,20){\line(1,1){10}}
\color{black}
\put(20,20){\line(-1,1){10}}

\color{blue}
\put(10,30){\line(1,1){10}}
\put(30,30){\line(-1,1){10}}
\put(-10,30){\line(1,1){10}}
\put(10,30){\line(-1,1){10}}
\color{black}

\put(30,30){\line(-1,1){10}}
\color{blue}
\put(10,30){\line(1,1){10}}
\put(20,40){\line(1,1){10}}
\put(40,40){\line(-1,1){10}}
\color{black}

\put(-80,0){\line(1,0){160}}
\put(0,-5){\makebox(0,0){$\textrm{\tiny{0}}$}}
\put(-10,-5){\makebox(0,0){$\textrm{\tiny{-1}}$}}
\put(-20,-5){\makebox(0,0){$\textrm{\tiny{-2}}$}}
\put(-30,-5){\makebox(0,0){$\textrm{\tiny{...}}$}}
\put(10,-5){\makebox(0,0){$\textrm{\tiny{1}}$}}
\put(20,-5){\makebox(0,0){$\textrm{\tiny{2}}$}}
\put(30,-5){\makebox(0,0){$\textrm{\tiny{...}}$}}

\end{picture}
\centering
\begin{picture}(200,0)(-100, -60)
\put(-80,-40){\line(1,1){10}}
\put(-80,-40){\line(-1,1){10}}
\color{green}
\put(-90,-30){\line(1,1){10}}
\put(-70,-30){\line(-1,1){10}}
\color{black}

\put(-20,-40){\line(1,1){10}}
\put(-20,-40){\line(-1,1){10}}
\put(-30,-30){\line(1,1){10}}
\put(-10,-30){\line(-1,1){10}}

\put(30,-40){\line(1,1){10}}
\put(30,-40){\line(-1,1){10}}
\color{red}
\put(20,-30){\line(1,1){10}}
\put(40,-30){\line(-1,1){10}}
\color{black}

\put(-50,-35){\makebox(0,0){$\stackrel{\textrm{\tiny{growth}}}{\longrightarrow}$}}
\put(60,-35){\makebox(0,0){$\stackrel{\textrm{\tiny{death}}}{\longrightarrow}$}}
\put(5,50){\makebox(0,0){$\longrightarrow$}}
\put(5,175){\makebox(0,0){$\longrightarrow$}}

\put(90,-40){\line(1,1){10}}
\put(90,-40){\line(-1,1){10}}
\end{picture}

\caption{The corner growth model with equal growth (in \color{green}green\color{black}) rate and death (in \color{red}red\color{black}) rate\protect\footnotemark.}\label{fig:Corner growth}
\setlength{\unitlength}{1.3pt}
\end{figure}
\footnotetext{Inspired by figure $0a$ in \cite{corwin2011kardarparisizhangreview}.}

In 1980, \cite{rost1981non} proved the following law of large numbers for the corner growth model, where the scaling is of the same scale in height, time and space. Note that the limiting profile obtained is parabolic. 

\begin{theorem}\label{thm: Wedge initial data}
    For wedge initial data, 
    \[
    \displaystyle \lim_{M\to\infty}\frac{h(tM,xM)}{M} = \mathfrak{h}(t,x)\coloneqq 
    \begin{cases} 
        t\frac{1-(x/t)^2}{2} & |x|<t\\
        |x| & |x|\geq t 
    \end{cases}
    \]
\end{theorem}

The fluctuations of the model around the limiting profile \(\mathfrak h\) are what is believed to be universal. Simulations, see \cite{corwin2016kardar} show that the fluctuations around \(\mathfrak h\) are small in comparison the transversal correlations, due to the parabolic nature of the limiting profile. This leads one to consider the centered height function 
\[
h_{\epsilon}(t,x)\coloneqq  \epsilon^bh(\epsilon^{-z}t, \epsilon^{-1}x)-\frac{\epsilon^{-1}t}{2} \quad \text{for } \epsilon>0.
\]
\noindent where the \textit{dynamic scaling exponent} \(z = 3/2\) and the \textit{fluctuation exponent} \(b = 1/2\). The above scaling corresponds to the aforementioned 3 :
2 : 1 scaling in time, space and fluctuations respectively. Note that these are the characteristic exponents for the KPZ universality class. That this scaling produces non-trivial results was shown in 1999, when Johansson \cite{johansson2000shape}, proved that for fixed \(t\), as \(\epsilon \to 0\), the random
variable \(h_{\epsilon}(t,0)\) converges to a GUE Tracy-Widom distributed random variable, which heavily features in the theory of random matrices. As for the entire scaled process \(h_{\epsilon}(t, x)\), convergence to a limiting process (the \textbf{KPZ fixed point}), with dependence on initial data, remains conjectural.\\

The dynamic nature of the KPZ fixed point, its dependence on initial data and the 3:2:1 scaling leads to the study of the KPZ equation, a stochastic PDE, which qualitatively seems to exhibit properties that characterise membership in the KPZ universality class, which we now turn to.

\subsection{KPZ equation and its scaling properties}
The \textbf{KPZ equation} is written as
\begin{equation}\label{eq: KPZ eqn}
    \frac{\partial h}{\partial t}(t,x) = \nu\frac{\partial^2 h}{\partial x^2}(t,x)+\frac{1}{2}\lambda \left(\frac{\partial h}{\partial x}\right)^2(t,x)+\sqrt{D}\xi(t,x)
\end{equation}

\noindent where \(\xi(t, x)\) is Gaussian space-time white noise, formally, a mean-zero Gaussian process with no spatial correlation, though technically this construction can only be carried through when \(\xi\) is seen as a \textit{stochastic tempered distribution}, see \cite{streit1994introduction}. Note \(\lambda, \nu \in \R, D > 0\) and \(h(t, x)\) is taken to be a real valued continuous function of time \(t \in \R_{+}\) and space \(x \in \R\). Due to the white-noise, it is natural to expect \(x \mapsto h(t, x)\) to be only as regular as Brownian motion. It is worth noting that in general this equation is classically ill-posed, see \cite{quastel2011introduction}, though Hairer developed the theory of regularity structures, see \cite{hairer2014theory}, as a means of constructing solutions to certain ill-posed SPDEs, particularly the KPZ equation, ultimately earning him the Fields medal.\\

Qualitatively, the equation contains the aforementioned properties  characterising membership in the KPZ universality class. This can be seen by analysing the terms that appear in the KPZ equation \ref{eq: KPZ eqn}. First, the growth is local, smoothing appears through the Laplacian term, non-linear slope-dependent growth is also expected to be present due to the square of the gradient, and the presence of Gaussian white-noise is compatible with the requirement that growth be driven by noise that decorelates quickly and is not heavy tailed.\\

The question still remains, how might one see the 3 : 2 : 1 scaling in the KPZ equation? We present the following heuristic due to \cite{corwin2016kardar}. One can proceed as before and define the scaled functions \(h_{\epsilon}(t,x)\coloneqq  \epsilon^bh(\epsilon^{-z}t, \epsilon^{-1}x)\) for \(\epsilon > 0\). Then, \(h_{\epsilon}(t,x)\) still satisfies the KPZ equation \ref{eq: KPZ eqn}, though with different with parameters, namely, \(\epsilon^{2-z}\nu, \epsilon^{2-z-b}\frac{1}{2}\lambda\) and \(\epsilon^{b-\frac{z}{2}+\frac{1}{2}}\sqrt{D}\). It turns out that two-sided Brownian motion is stationary for the KPZ equation; thus, it is natural to impose that $b = 1/2$, so that the scaling respects the Brownian scaling of the initial data. With this constraint, the only way to avoid coefficients blowing up to infinity,
or shrinking down to zero (as \(\epsilon \to 0\)) is to
choose \(z = 3/2\). The above informal argument makes a convincing case for the plausibility of the 3 : 2 : 1 making its appearance in the class of models, whose dynamics are believed to be captured by the KPZ equation, as first postulated by Kardar, Pharisi and Zhang in \cite{kardar1986dynamic}.\\

\section{KPZ fixed point}

The random growth processes that are expected to have the same scaling and asymptotic fluctuations as the KPZ equation and converge to the universal limiting object called the KPZ fixed point, are said to lie in the KPZ universality class. Though, the exact nature of this limiting KPZ fixed point remains rather elusive. The above heuristic arising from a qualitative analysis of the KPZ equation reveals essentially no new information on the distribution of this limiting process.\\

More recently however, in the realm of interacting particle systems, the authors in \cite{matetski2021kpz}, constructed a candidate for the KPZ fixed point upon considering the scaling limit of the associated height function of the totally asymmetric simple exclusion process (TASEP), under the 3:2:1 scaling, and proved that it has a Markovian structure in time. Briefly, in this model, one considers an infinite sequences of particles with positions
\[
\dots<X_t(2)<X_t(1)<X_t(0)<X_t(-1)<X_t(-2)<\dots \quad \text{ on } \Z\cup\pm \infty.
\] 
The evolution of the system is driven by each particle independently attempts to occupy its right-adjacent spot, provided it is \textit{empty}, with exponential waiting time with rate \(1\), see figure \ref{fig:TASEP}. To each integer \(u\), one can associate \(X^{-1}_t(u) = \min\{k\in \Z: X_t(k)\leq u\}\), that the label of the right-most particle with position up to \(u\). 

\begin{figure}[H]
    \centering
    \includegraphics[scale = 0.5]{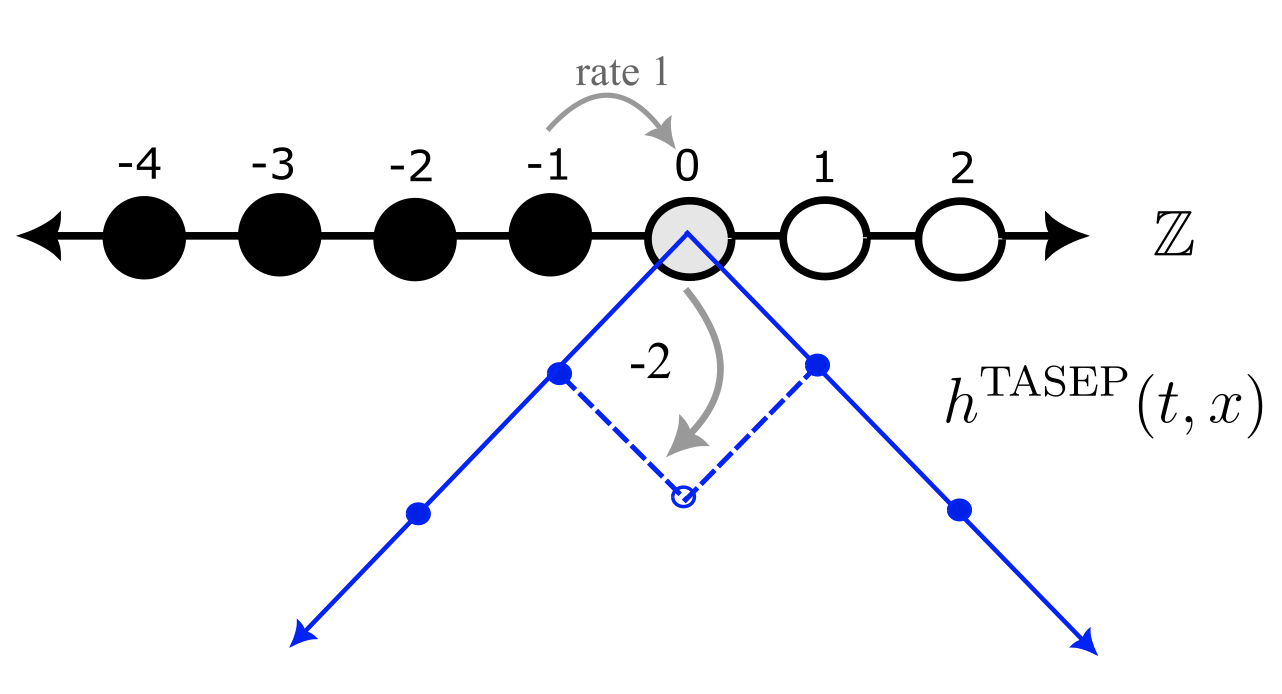}
    \caption{Height function for TASEP with step initial condition. The convention is to label the rightmost particle as \textquotedblleft\(1\)\textquotedblright. Occupied spots are denoted in black, and empty ones in white. Observe that the rightmost particle moving to the right (indicated with \color{black!60} grey\color{black}) makes the local maximum into a local minimum in the height function.}
    \label{fig:TASEP}
\end{figure}

\noindent The height function at time \(t\in\R_{+}\) is now given by
\[
h^{\text{TASEP}}(t,z) = -2(X^{-1}_t(z-1))-X^{-1}_0(-1)-z, \quad \text{ for } z\in\Z
\]
and interpolating linearly in between, so that \(h^{\text{TASEP}}(0,0) = 0\). Further observe that 
\[
h^{\text{TASEP}}(t, z+1) = h^{\text{TASEP}}(t, z) + \eta_t(z), \quad z\in \Z,
\]
with 
\[
\eta_t(z) = 2(X^{-1}_t(z-1)-X^{-1}_t(z)) -1\]
\[
= \begin{cases}
    1 & \text{if there is a particle at } z \text{ at time } t\\
    -1 & \text{if there is no particle at } z \text{ at time } t
\end{cases}
\]
which means that a local maximum of \(h^{\text{TASEP}}_t\) at \(z\) will evolve into a local minimum at \(z\) with rate 1, and the rest of the function remains unchainged. This is because, for a local maximum to exist at \(z\), we require \(h^{\text{TASEP}}(t, z+1) = h^{\text{TASEP}}(t, z)+1\), and \(h^{\text{TASEP}}(t, z-1) = h^{\text{TASEP}}(t, z)-1\) and so by the above, there must be a particle at \(z-1\) and no particle at \(z\). Now, if one makes the KPZ 3:2:1 scaling (including shifting)
\begin{equation}\label{eq: KPZ scaling}
h_{\epsilon}(t,x) = \epsilon^{1/2}(h^{\text{TASEP}}(2\epsilon^{-3/2}t, 2\epsilon^{-1}x) + \epsilon^{-3/2}t), \quad t>0\,,
\end{equation}
Matetski, Remenik annd Quastel in \cite{matetski2021kpz} proved that that \(h_{\epsilon}\) has a distributional limit in the space of upper-semicontinuous functions with certain growth conditions, equipped with the topology of local uniform convergence (see \ref{eq: UC}), which they call the KPZ fixed point.\\

From the construction of the TASEP model, modeling its evolution by taking independent Poisson processes of rate 1 at every site on \(\Z\) and when the Poisson process jumps, the height function jumps down by 2 if and only if it is a local maximum. This is in line with there being an available spot next to the particle, keeping in mind the above discussion. Thus, the evolution of the height given different initial conditions \(h_0\) is coupled, with the coupling depending on the underlying Poisson process construction of the model driving growth. If one considers the height function at time one
\[
\mathcal{A}^{\text{TASEP}}(x,y) = h(1,y), \quad \text{ with } h_0(\cdot) = -|\cdot-x|, x,y\in\R
\]
and its 3:2:1 scaled version, \(\mathcal{A}_\epsilon\), with the same scaling as above for the height function, it has been shown in \cite{matetski2021kpz} that the laws of the \(\mathcal{A}_\epsilon\) are tight in the space \(\mathcal{C}(\R^2; \R)\) for \(\epsilon> 0\). 
Any limiting process is called an \textbf{Airy sheet}, and appears as a
\textquotedblleft residual forcing noise. A variational characterisation of the fixed point was given in \cite{matetski2021kpz} in terms of subsequential limits of these Airy Sheets, $\mathcal{A}$ with initial data \(h_0(\cdot)\)
\begin{equation}\label{eq: variation airy}
    h(t, x) = \sup_{y\in\R}\Bigl(t^{1/3}\hat{\mathcal{A}}(t^{-2/3}x, t^{-2/3}y)+h_0(y)\Bigr)
\end{equation}
in law, where \(\hat{\mathcal{A}}(x,y)=\mathcal{A}(x,y)-(x-y)^2\) and satisfies the metric composition law \ref{eq: metric Airy}. Uniqueness of any subsequential limits was not shown until a proof was provided in the seminal work of \cite{DOV}, see \ref{sec:Airy sheet}, thereby nailing down the right hand side of equation \ref{eq: variation airy}.  

\chapter{KPZ models}
\section{Last Passage Percolation}
\indent We now come to one of the key objects in the KPZ universality class, namely, Brownian last passage percolation which was integral in the construction of the Directed Landscape and Airy Sheet. We begin with the collection of some preliminary facts regarding last passage percolation over ensembles of functions following \cite{DOV}.\\

Formally, let \(I\subset \mathbb{Z}\) be a possible finite index set and define the space \(C^{I}\) of sequences of continuous functions with domain the reals. That is the space
\[f: \mathbb{R}\times I\to \mathbb{R}\quad (x,i)\mapsto f_i(x)\,.\]

\begin{definition}[Path] Let \(x\leq y\) in \(\mathbb{R}\), and \(m\leq \ell\) in \(\mathbb{Z}\) respectively. A \textbf{Path}, from \((x,\ell)\) to \((y,m)\) is a non-increasing function \(\pi: [x,y] \to \mathbb{N}\) which is cadlag on \((x,y)\) and takes the values \(\pi(x)= \ell\) and \(\pi(y)= m\).
\end{definition}

\begin{remark}
    The convention that the paths be non-decreasing is so that they match the natural indexing of the Airy Line ensemble, see section \ref{Airy Line}. Thus, when visualised in the plane, the paths become non-decreasing, see figure \ref{fig:Paths}.
\end{remark}

In what is to follow, since we will primarily be considering the Airy line ensemble (see \ref{Airy Line}), hence we will take the indexing set to be \(I = \mathbb{N}\). We now define an important quantity associated to each such path, namely, its \textit{length} as the sum of increments of \(f\) along \(\pi\). This also leads one to natural define a derived quantity, namely the \textit{last passage value}.

\begin{definition}[Length] Let \(x\leq y \in \mathbb{R}\) and \(m < \ell\in\mathbb{Z}\). For each \(m\leq i <\ell\), let \(t_i\) denote the jump of the path \(\pi\), on an ensemble \((f_i)_{i\in I}\), from \(f_{i+1}\) to \(f_i\). Then the length of \(\pi\) is defined as 
\[
\ell(\pi) = f_m(y)-f_m(t_m) + \displaystyle\sum_{k = m+1}^{\ell-1}(f_i(t_{i-1})-f_i(t_i))+f_{\ell}(t_{\ell-1})-f_{\ell}(x)\,.
\]
\end{definition}

\begin{definition}[Last Passage Value]\label{Definition: last passage}
    With \(x\leq y, m<\ell\) as before and \(f\in C^{I}\), define the \textbf{last passage value} of \(f\) from \((x,\ell)\) to \((y,m)\) as
    \[
    f[(x,\ell)\to(y,m)] \coloneqq \displaystyle \sup_{\pi}\ell(\pi)\,,
    \]
where the supremum is over precisely the paths \(\pi\) from \((x,\ell)\) to \((y,m)\).
\end{definition}
\begin{remark}
    Any path \(\pi\) from \((x,\ell)\) to \((y,m)\) such that its length is equal to its last passage value is called a \textbf{geodesic}. Such geodesics possess a rich structure that was investigated in depth in both \cite{DOV} and \cite{sarkar2021brownian} and will be discussed further below.
\end{remark}

Note that the length of a path \(\ell(\pi)\), can be viewed as a function on the subset \(\mathcal{Z}\) of non-increasing cadlag functions with fixed endpoints in \(\mathbf{D}\), the space of cadlag functions \(\mathbf{D}\coloneqq \mathbf{D}([x,y], \N)\). When endowed with respect to to the  Skorokhod topology, which is metrisable, with metric given by \[
d(f,g) = \displaystyle \inf_{\lambda\in \Lambda}\max\{\|\lambda-id\|,\|f\circ\lambda-g\|\}, \quad \text{for } f,g \in \mathbf{D}\,,
\] 
where \(\Lambda\) is the set of strictly increasing continuous functions from \([x,y]\) to \([x,y]\) such that \(\lambda(x) = x\) and \(\lambda(y) = y\), the above function is continuous. Indeed, this can be seen by viewing \(\ell(\pi)\) as a sum of increments \(f\in \mathcal{C}^I\) over a path \(\pi\). Formally, we use the following lemma to help us establish continuity. 

\begin{lemma}
    For \(x_n, x\in \mathbf{D}\), if \(d(x_n,x)\to 0, n\to \infty\), and \(x\) is non-constant, then \(x_n(\tau_n)\to x(\tau), n\to \infty\), where \(\tau_n, \tau\) are any jump times of \(x_n, x\) respectively (for \(n\) large enough). 
\end{lemma}

\begin{proof}
    Since \(d(x_n,x)\to 0,\) as \(n \to \infty\), there exists a sequence \((\lambda_n)_{n\in \N}\) of continuous, strictly increasing functions in \(\Lambda\) such that 
    \begin{equation}\label{eq: metric}
        \|\lambda_n-\text{id}\|_{\infty}\to 0, \quad \|x_n\circ\lambda_n-x\|_{\infty}\to 0, \text{as } n\to \infty.
    \end{equation}
    In particular, there exists \(N\in \N\) such that
    \[
    \|x_n\circ \lambda_n - x\|_{\infty}<1, \quad \text{for all } n\geq N.
    \]
    Since \(x_n, x\) take values in \(\N\), it follows that 
    \begin{equation}\label{eq: reparam}
        \forall n\geq N \forall t\in [x,y]: \quad  x_n(\lambda_n(t))= x(t).
    \end{equation}
    
\noindent Thus, up to a strictly increasing time reparameterisation, the functions \((x_n)_{n\geq N}\) agree with \(x\). Suppose now that \(\tau_i\) is the \(i\)th jump time of \(x\). It follows from \ref{eq: reparam} and a simple induction that for \(n\geq N\), \(x_n\) has the same number of jumps as x with \(\lambda_n(\tau_i)\) being \(i\)th jump time of \(x_n\). This in conjunction with \ref{eq: metric} concludes the proof.
\end{proof}

Since \(\mathcal Z\) is closed with respect to the above topology of \textquotedblleft jump times\textquotedblright, a compactness argument using Arzela-Ascoli, see Chater $3$ in \cite{billingsley2013convergence}, implies that the supremum over admissible paths is attained, showing that there do indeed exist geodesics. \\

Last passage percolation enjoys the following \textbf{metric composition law}, Lemma 3.2 in DOV \cite{DOV}.

\begin{lemma}[Metric composition law]\label{Lemma: Metric Composition}
    Let \(x\leq y \in \mathbb{R}\), \(m < \ell\in\mathbb{Z}\) and \(f\in C^I\). If \(k\in \{m, \dots, \ell\}\), then we have
    \[
    f[(x,\ell)\to(y,m)] = \displaystyle \sup_{z\in[x,y]}(f[(x,\ell)\to(z,k)]+f[(z,k)\to(z,m)])
    \]
    and if \(k\in \{m+1, \dots, \ell\}\), then 
    \[
    f[(x,\ell)\to(y,m)] = \displaystyle \sup_{z\in[x,y]}(f[(x,\ell)\to(z,k)]+f[(z,k-1)\to(z,m)])
    \]
    Furthermore for any \(z\in [x,y]\), 
    \begin{equation}\label{eq: composition}
    f[(x,\ell)\to(y,m)] = \displaystyle \sup_{k\in \{m, \dots, \ell\}}(f[(x,\ell)\to(z,k)]+f[(z,k)\to(z,m)])
    \end{equation}
\end{lemma}
\begin{proof}
These statements are deterministic properties of last passage and for the sake of brevity, we will only focus on proving \ref{eq: composition}, as the rest follow in a similar manner.\\
\indent (\(\leq\)): fix \(z\in[x,y]\) and observe that for any path \(\pi\) from \((x,\ell)\) to \((y,m)\), its length \(\ell(\pi)\), that is the sum of increments of \(f\) over \(\pi\), can be split into two terms corresponding to the lengths of the restrictions of \(\pi\) to \([x,z]\) and \([z,y]\), viewed as paths, and thus we obtain the bound on the right hand side of \ref{eq: composition}. Taking suprema over \(\pi\) concludes the proof of this direction.\\
\indent (\(\geq\)): notice that for any \(k\in\{m, \dots, \ell\}\)
and paths \(\pi_1\) from \((x,\ell)\) to \((z,k)\) and \(\pi_2\) from \((z,k)\) to \((y,m)\), they can either:
(i) be concatenated to produce a path \(\pi\) from \((x,\ell)\) to \((y,m)\) if \(\pi_2\) is right-continuous at \(z\)
or
(ii) construct a path that agrees with \(\pi_1\) on \([x,z)\) and \(\pi_2\) on \((z, y]\) and is equal to the right limit of \(\pi_2\) at \(z\), which always exists due to the fact that all paths are non-increasing. In both cases, the length of the constructed path from \((x,\ell)\) to \((y,m)\) has length equal to the sum of the lengths of \(\pi_1\) and \(\pi_2\), which is of course bounded by the last passage value. Taking suprema over all admissible \(\pi_1, \pi_2\) and \(k\in \{m,\dots, \ell\}\), enables us to conclude the proof of equality. 
\end{proof}

We now collect an immediate corollary.
\begin{corollary}[Reverse triangle inequality]
    For any \(z\in [x,y]\) and \(k\in \{m, m+1,\dots, \ell\}\), 
    \[
    f[(x,\ell)\to(y,m)] \geq \displaystyle f[(x,\ell)\to(z,k)]+f[(z,k)\to(z,m)]
    \]
\end{corollary}
    
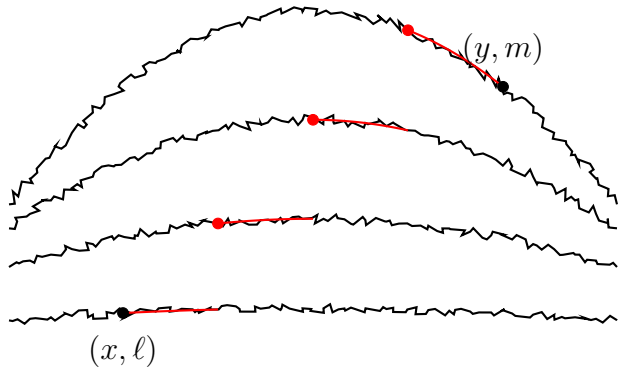
\begin{figure}[H]
  \centering
      \begin{tikzpicture}
    \pgfmathsetseed{12345}

    \def\startx{1}
    \def\endx{9}
    \def\startg{2.5}
    \def\endg{7.5}
    \def\spacing{1} 


    \foreach \i in {1,...,4} {

      \draw[thick, black, decoration={random steps, segment length=2pt, amplitude=2pt}, decorate] 
      plot[domain=\startx:\endx, samples=30] (\x, {\i*\spacing - 2.5 - (\i)*(\i)*(\x-\startg)*(\x-\endg)*0.01});
      
      \draw[thick, red, decoration={random steps, segment length=2pt, amplitude=2pt}, decorate]
      plot[domain=\startg+(\i-1)*(\endg-\startg)/4:\startg+(\i)*(\endg-\startg)/4, samples=30] (\x, {\i*\spacing - 2.5 - (\i)*(\i)*(\x-\startg)*(\x-\endg)*0.01});
      \ifnum\i>1
      \ifnum\i<5
        \filldraw[red] (5/4*\startg+\i*0.25*\endg-\i*0.25*\startg-0.25*\endg, {\i*\spacing - 2.5 - (\i)*(\i)*((\i-1)/4)*(\endg-\startg)*(\startg+(\i-1)/4*(\endg-\startg)-\endg)*0.01}) circle (2pt);
      \fi
    \fi

        } 
    \filldraw[black] (\startg, 1*\spacing-2.5
      ) circle (2pt);
    \filldraw[black] (\endg, 4*\spacing-2.5
      ) circle (2pt);

      \node at (\startg, 1*\spacing-3) {\((x,\ell)\)};
      \node at (\endg, 4*\spacing-2) {\((y,m)\)};
   \end{tikzpicture}
   
   \caption{Visualisation of a possible path (red) \textquotedblleft embedded \textquotedblright \space on the Airy line ensemble, here \((\mathcal{A}_1, \mathcal{A}_2, \mathcal{A}_3, \mathcal{A}_4)\) and \(m = 1, \ell = 4\) (see section \ref{Airy Line}).}
  \label{fig:Paths}
\end{figure}

We adopt the following important definition from \cite{corwin2014brownian}, since from now on, we will be concerning ourselves with primarily not deterministic, but random ensembles.

\begin{definition}[Random Ensemble]\label{def; random ensemble}
    Let $\Sigma$ be a (possibly infinite) interval of $\Z$, and let $\Lambda$ be an interval of $\R$. Consider the set \(X\coloneqq\mathcal{C}^\Sigma\) of continuous functions $f:\Sigma\times \Lambda \rightarrow \R$. We endow it with the topology of uniform convergence on compact subsets of $\Sigma\times\Lambda$. Let $\mathcal{C}$ denote the sigma-field  generated by Borel sets in $X$.\\

    A {\it $\Sigma$-indexed line ensemble} $\mathcal{L}$ is a random variable defined on a probability space $(\Omega,\mathcal{B},\PP)$, taking values in $X$ such that $\mathcal{L}$ is a $(\mathcal{B},\mathcal{C})$-measurable function. Furthermore, we write $\mathcal{L}_i:=(\mathcal{L}(\omega))(i,\cdot)$ for the line indexed by $i\in\Sigma$. 
\end{definition}

\begin{remark}
\centering
\begin{enumerate}
    \item The extended reals $\R\cup\{-\infty,+\infty\}$ will be denoted by $\R^*$ and we endow $\R^*$ with the usual topology of $\R$ and the discrete topology for $+\infty$ and $-\infty$.
    \item Note the sigma algebra \(\mathcal{C}\) is generated by the cylinder sets $\{f\in X: f(i,x)\in U\}$, as the parameters range over $i\in \Sigma$, $x\in \Lambda$ and $U \subseteq \R^*$ open.
    \item We will often slightly abuse notation and write $\mathcal{L}:\Sigma\times \Lambda \rightarrow \R$, even though it is not $\mathcal{L}$ which is such a function, but rather $\mathcal{L}(\omega)$ for each $\omega \in \Omega$. 
\end{enumerate}
\end{remark}

Intuitively, $\mathcal{L}$ is a collection of random continuous curves (even though we use the word \textquotedblleft line \textquotedblright we are referring to continuous curves), indexed by $\Sigma$, each of which maps $\Lambda$ into $\R$.\\

Given a $\Sigma$-indexed line ensemble $\mathcal{L}$, and a sequence of such ensembles $\big\{ \mathcal{L}^N: N \in \N \big\}$,  a natural notion of convergence is the weak-* convergence of the measure on $(X,\mathcal{C})$ induced by $\mathcal{L}^N$, to the measure induced by $\mathcal{L}$ made precise below.

\begin{definition}[Weak Convergence of Line Ensembles]\label{def: convergence in law}
    With $\mathcal{L}$, and $\big\{ \mathcal{L}^N: N \in \N \big\}$ as above, we say $\big\{ \mathcal{L}^N: N \in \N \big\}$ \textbf{converge weakly as line ensembles} to $\mathcal{L}$ if for all bounded continuous functionals $F$, 
    \[
    \int d\PP(\omega) F(\mathcal{L}^{N}(\omega)) \to \int d\PP(\omega) F(\mathcal{L}(\omega)), \quad \text{as } N\to \infty.
    \]
    and denote it by $\mathcal{L}^N\Rightarrow \mathcal{L}$.
\end{definition}

Furthermore, a line ensemble is said to be \textbf{non-intersecting} if, for all $i<j$, $\mathcal{L}_i(r)>\mathcal{L}_j(r)$ for all $r\in \Lambda$. All statements are to be understood as being almost sure with respect to $\PP$.\\

An example of interest, which has been intensely studied in the literature is the case when one considers a sequence of independent Brownian motions \(B = (B_1, B_2, \dots)\) as an element of the ensemble space \(C^{\mathbb{N}}\), and defines the resulting last passage values \(B[(0,n) \to (1,1)]\) and geodesics \(\pi_n\). These geodesics, have non-trivial scaling limits, namely the content of Theorem 1.1 in \cite{DOV} which states that under a suitable coupling, they converge almost surely to the so called \textquotedblleft directed geodesic \textquotedblright. They are inextricable linked to the Directed Landscape, a random element constructed as a scaling limit of Brownian Last passage percolation, see section \ref{Directed Landscape}. 

\subsection{Pitman Transform}

\indent Having now given the general framework of last passage percolation on ensembles of continuous functions, in the following section, we will define an operation on pairs of continuous functions, namely, the Pitman transform, that essentially produces two non-intersecting curves and can be expressed in terms of last passage percolation. Keep in mind that the proof of the absolute continuity of the KPZ fixed point will be reduced to the case of Brownian last passage percolation, in this sense, the tools developed here play a crucial role, see figure \ref{fig:flowchart}.\\

\indent We will now introduce again certain operations on the space of pairs of continuous functions, denoted by \(\mathcal{C}^2_{+}\) and elements \(f = (f_1, f_2)\), where the \(f_i, i = 1,2\) are real-valued continuous functions on the positive reals. For \(f\in\mathcal{C}^2_{+}\), we are going to construct the Pitman transform \(Wf = (Wf_1, Wf2)\in \mathcal{C}^2_{+}\) in the following way. 

\begin{definition}[Pitman Transform]\label{Pitman}
    For distinct \(x<y\in [0,\infty )\), define their maximal gap size as
\[
G(f_1, f_2)(x,y) \coloneqq  \max\{\displaystyle \max_{s\in[x,y]}(f_2(s)-f_1(s)), 0 \}\,.
\]
Then set
\[
Wf_1(t) = f_1(t)+G(f_1, f_2)(0,t) 
\]
and
\[
Wf_2(t) = f_2(t)-G(f_1, f_2)(0,t)\,.
\]
\end{definition}

\noindent
\begin{minipage}{0.6\textwidth}  
$\quad$ Essentially, the curve \(Wf_1\), represents the reflection of \(f_1\) off \(f_2\), see figure \ref{fig:Paths} for a more concrete example.  Furthermore,  the fact that 
$G(f_1, f_2)(x,y) =  \displaystyle \max_{s\in[x,y]}(f_2(s)-f_1(s))$ when $ \displaystyle \max_{s\in[x,y]}(f_2(s)-f_1(s))\geq 0 $ and zero otherwise,  gives that \(Wf_1\geq \max\{f_1, f_2\}\) and  \(Wf_2\leq \min\{f_1, f_2\}\).\\

$\quad$ More interesting perhaps is that one can express the top line of the Pitman transform in terms of last passage values. 
\end{minipage}
\hfill  
\begin{minipage}{0.35\textwidth}  
    \centering
   \begin{tikzpicture}[scale = 0.7]
  \draw[black, thick, domain=0:2, samples=20] plot (\x, {0.1*\x^4 + 1}) node[right] {\(f_1\)};
  
  \draw[black, thick, domain=0:2, samples=20] plot (\x, {2*exp(-\x)}) node[right] {\(f_2\)};

  \draw[red, thick, dotted, domain=0:2, samples=20] plot (\x, {0.1*\x^4 + 2}) node[right] {\(Wf_1\)};

  \draw[blue, thick, dotted, domain=0:2, samples=20] plot (\x, {2*exp(-\x)-1}) node[right] {\(Wf_2\)};

  \draw[->] (-0.2,0) -- (2.7,0) node[right] {\(x\)};
  \node[below] at (0,-0.1) {0};
  \node[below] at (2,-0.1) {2};

\end{tikzpicture}
    
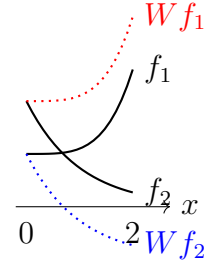
\captionof{figure}{Pitman transform of two simple continuous curves, \(f_1 = 0.1x^{4}+1, f_2 = 2e^{-x}\) on the interval \([0,2]\). Notice  \(G(f_1, f_2)\equiv 1\).}
    \label{fig:pitman}
\end{minipage}

\begin{lemma}\label{lemma: Pitman melon}
    Let \(f\in C^2_+\) and let \(Wf = (Wf_1, Wf_2)\) be as above. Then for all \(t\in[0,\infty)\),
    \[
    Wf_1(t) = \displaystyle \max_{i=1,2}\{f_i(0) + f [(0, i) \to (t, 1)]\} \,.
    \]
\end{lemma}

\begin{proof}
    By definition, 
    \[
    Wf_1(t) = f_1(t)+G(f_1, f_2)(0,t)\]
    \[
    = f_1(t) + \max\{\displaystyle \max_{s\in[x,y]}(f_2(s)-f_1(s)), 0 \}
    \]
    \[
    = \max\{\displaystyle \max_{s\in[x,y]}(f_2(s)+f_1(t)-f_1(s)), f_1(t) \}\,.
    \]
    Recall from \ref{Definition: last passage} and \ref{eq: composition} the definition of last passage values and the metric composition law respectively. This gives \(f_1(t) = f_1(0) + f [(0, 1) \to (t, 1)]\) and \[ \displaystyle \max_{s\in[0,t]}(f_2(s) + f_1(t) - f_1(s))\} = f_2(0) + f [(0, 2) \to (t, 1)]\,.\] Combining the above gives the result. 
\end{proof}

Particularly in the case where \(f_1(0) = f_2(0) = 0\), we obtain that 
\[
Wf_1(t) = f [(0, 2) \to (t, 1)]\,.
\]
\(Wf\) is commonly referred to the \textbf{Melon} of f, since paths in \(Wf\) avoid each other and thus resemble the stripes of a watermelon.\\

An application of the above that is of interest is that of two independent standard Brownian motions (starting from zero) \(B = (B_1, B_2)\).  Let \(\hat{B}=(\hat{B}_1, \hat{B}_2)\) be two Brownian motions conditioned not to collide, in the sense of Doob \cite{sarkar2021brownian}. Then, the law of the melon \(WB\) is the same as that of \(\hat{B}\). In \cite{o2002representation}, a generalisation was proved for \(n\) Brownian motions, call it \(WB^n\). Furthermore, after appropriate rescaling, \(WB^n\) converges in law to a non-intersecting ensemble on \(C^{\mathbb{N}}\) (with respect to the product on uniform topology on \(C^{\mathbb{N}}\)), Theorem 2.1 in \cite{DOV}, which we will now discuss. 

\section{Airy Line Ensemble}\label{Airy Line}

\begin{theorem}\label{thm: Melon scaling}
    Let \(WB^n\)  be a Brownian \(n\)-melon. Define the rescaled melon \(A^n = (A^n_1, \dots, A^n_n)\)
by
\[
A^n_i(y) = n^{1/6} \left((WB^n)_i(1 + 2yn^{-1/3}) - 2\sqrt{n} - 2yn^{1/6} \right).
\]
Then \(A^n\) converges to a random sequence of functions \(\mathcal{A} = (\mathcal{A}_1, \mathcal{A}_2, \dots) \in C^\mathbb{N}\) in law with respect to product of uniform-on-compact topology on \(C^\mathbb{N}\). For every \(y \in \mathbb{R}\) and \(i < j\), we have that \(\mathcal{A}_i(y) > \mathcal{A}_j(y)\). The function \(\mathcal{A}\) is called the \textbf{(parabolic) Airy line ensemble.}
\end{theorem}

\begin{figure}[H]
    \centering
    \begin{tikzpicture}[
    block/.style={draw, thick, minimum height=1.5cm, minimum width=1.5cm}]
    
    \draw[black, thick] (0,0) -- (0.5,1) -- (0.75,0.7) -- (1,1.3) -- (1.5,1.2) -- (2,2.4) -- (2.5, 2.7) -- (3, 3) node[above right] {\(WB_1\)};
    \draw[black, thick] (0,0) -- (0.5,0.7) -- (0.75,0.5) -- (1,0.5) -- (1.5,1) -- (2,1.4) -- (2.5, 2) -- (3, 1.5) node[above right] {\(WB_2\)};
    \draw[black, thick] (0,0) -- (0.5,-0.3) -- (1,0.1) -- (1.5,0.5) -- (1.75, 0) -- (2,-0.2) -- (2.5, 0.8) -- (2.75, 0.3) -- (3, -0.3) node[above right] {\(WB_3\)};
    \draw[black, thick] (0,0) -- (0.5,-1) -- (1,-0.3) -- (1.5,-0.2) -- (2,-0.7) -- (2.5, -1.2) -- (3, -1) node[above right] {\(WB_4\)};
    
    \begin{scope}[on background layer]
    \draw[blue, thick, fill=blue!10] (1, 0.8) -- (2.5, 1.3) -- (2.5, 2.3) -- (1, 1.8) -- cycle;
    \draw[blue, thick] (0.8,0.8) -- (0.8,1.8) node[midway, left] {\(O(n^{-1/6})\)};
    \draw[blue, thick] (0.75,1.8) -- (0.85,1.8);
    \draw[blue, thick] (0.75,0.8) -- (0.85,0.8);

    \draw[blue, thick] (1,2) -- (2.5,2.5) node[midway, above, pos = 0.7] {\(O(n^{-1/3})\)};
    \draw[blue, thick] (1,1.95) -- (1,2.05);
    \draw[blue, thick] (2.5,2.45) -- (2.5,2.55);
    
    \end{scope}

    \end{tikzpicture}
    \caption{Brownian Melon Scaling limit. Above is a realisation of of the \(WB^n, n=4\) melon. \textquotedblleft Zooming in \textquotedblright \space on the parallelogram with at small scales and taking the limit as \(n\to\infty\) yields the convergence in law to the (parabolic) Airy Line Ensemble.}
    \label{fig:Brownian melon}
\end{figure}
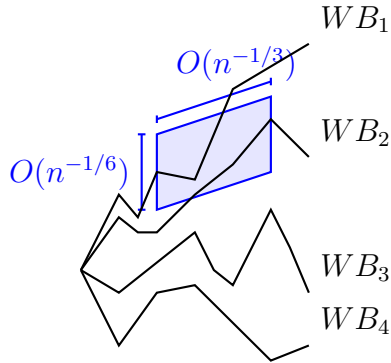

The shifted top-line of the Airy Line ensemble, \(\mathcal{A}_1(t)+t^2\) is a rather well-known process. It is stationary in time and for all \(t\in\R\), has the \textit{GUE Tracy-Widom distribution}, which is also shared with the asymptotic distribution of the rescaled height function \(h_{\epsilon}(s,0)\) for any \(s\) as \(\epsilon \to 0\) from section \ref{sec: non-gaussian}. This means that it satisfies certain explicit tail bounds, see Theorem 2.3 in \cite{DOV}.\\

Furthermore, Theorem \ref{Airy Line} shows that \(\mathbf{\mathcal{A}}_1\) is the scaling limit of the top line of a Brownian melon of \(n\) independent Brownian motions as \(n\to \infty\). Notice the spatial scaling is of order \(n^{1/3}\) and the fluctuation scaling is of order \(n^{1/6}\). This and the fact that is appears as a scaling limit of interfaces of random growth models starting from a single point \cite{sarkar2021brownian}, shows that the Airy process exhibits KPZ universality class traits.

\subsection{Brownian Gibbs Property}

\indent Below we restate the important \textbf{Brownian Gibbs Property} from \cite{sarkar2021brownian}, which elucidates the aforementioned locally Brownian nature of the Airy Line ensemble. 

\begin{theorem}[Brownian Gibbs Property]\label{thm: Brownian Gibbs}
Let \(\mathcal{A}\) denote the parabolic Airy line ensemble. For any \(k, \ell, \in\{0,1,2,\dots\}\) and \(a<b\in\mathbb{R}\), let \(\mathcal{F}\) denote the sigma algebra generated by 
\[\{\mathcal{A}_{i}(x):(i,x)\notin \{k+1, k+2, \dots, k+\ell\}\times (a,b)\}\]
Then the conditional distribution of \(\mathcal{A}_{|\{k+1, k+2, \dots, k+\ell\}\times [a,b]}\) given \(\mathcal{F}\) is given by the law of \(\ell\) independent Brownian bridges \(B_1, B_2, \dots, B_{\ell}:[a,b]\mapsto \mathbb{R}\) with diffusion parameters \(2\) such that \(B_{i}(a) = \mathcal{A}_{k+i}(a)\) and \(B_{i}(b) = \mathcal{A}_{k+i}(b)\) for all \(i = 1,2,\dots, \ell\) conditioned on the event
\[\mathcal{A}_k(r)>B_1(r)>B_2(r)>\dots >B_\ell(r)>\mathcal{A}_{k+\ell+1}(r)\]
for all \(r\) in \([a,b]\). For the above to hold in the case where \(k=0\), set \(\mathcal{A}_0\equiv \infty\).
\end{theorem}

\begin{remark}
    The above theorem enables for a more refined probabilistic understanding of the Airy Line ensemble. Indeed, one can quickly observe that this implies, in greater generality even, that any process satisfying the above property, particularly the Airy line ensemble is non-intersecting. Furthermore, the absolute continuity of the Airy Line ensemble with respect to the law of a Brownian Bridge with diffusion parameter \(2\) is immediate upon conditioning with respect to \(\mathcal{F}\). For a visualisation, see figure \ref{fig: Gibbs}.
\end{remark}

\begin{figure}[H]
  \centering
  
  \includegraphics[width = .5\linewidth]{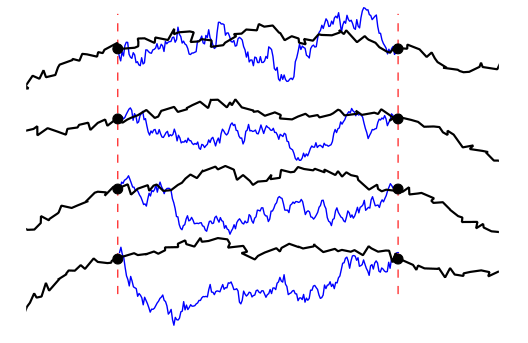}

  \caption{Visualisation of Brownian Gibbs Property on the first four lines of the parabolic Airy Line ensemble \(\mathcal{A} = \{\mathcal{A}_1>\mathcal{A}_2>\dots\}\) (in \textbf{black}) between two points (indicated by the \color{red}red \color{black} vertical dashed lines), indicated by the vertical dashed lines. The \color{blue}blue \color{black} curves are represent resampled versions of first four lines in the ensemble between the endpoints, conditioning on everything else.}
  \label{fig: Gibbs}
\end{figure}

\chapter{Directed Landscape}\label{Directed Landscape}

The Directed Landscape is the central object constructed in the DOV paper, \cite{DOV}, and its construction is in some sense, that will be elucidated below, analogous to that of Brownian motion. The underlying motif is that all of the objects constructed below are scaling limits in one way or another of Brownian last passage percolation, which was shown by \cite{nica2020one} to converge to the KPZ fixed point of Matetski, Quastel and Remenik \cite{matetski2021kpz}.

\section{Airy Sheet}\label{sec:Airy sheet}

The first step in constructing the Directed Landscape, is to start by constructing the Airy Sheet, which is stipulated to have the following characteristics.

\begin{definition}[Airy Sheet]\label{def: Airy Sheet}
    The \textbf{Airy Sheet} is a random continuous function \(\mathcal{S}:\R^2\to \R\) so that the following hold
    \begin{enumerate}
        \item Stationarity: \(\mathcal{S}\) has the same law as \(\mathcal{S}(\cdot+t, \cdot+t)\) for all \(t\in \R\)
        \item Coupling: \(\mathcal{S}\) can be coupled with a parabolic Airy line ensemble so that \(\mathcal{S}(0,\cdot) = \mathcal{A}_1(\cdot)\) and for all rational triplets \((x,y,z)\in \Q^{+}\times\Q^2\) there exist random integers \(K_{x,y,z}\) such that for all \(k\geq K_{x,y,z}\)
        \begin{equation}\label{eq: coupling}
        \mathcal{A}[(-\sqrt{k/(2x)}, k) \to (z,1) ] -  \mathcal{A}[(-\sqrt{k/(2x)}, k) \to (y,1) ]= \mathcal{S}(x, z) - \mathcal{S}(x, y).
        \end{equation}
    \end{enumerate}
\end{definition}

The coupling \ref{eq: coupling} will play a crucial role in later sections, when everything will be pieced together to prove the absolute continuity of the KPZ fixed point. The stationarity property and its coupling with last passage percolation over the Airy sheet are sufficient to characterise its law uniquely. Furthermore, one can construct the Airy Sheet as a limit of Brownian last passage percolation, as is done in \cite{DOV} by constructing the sequence of random functions from the following
\[
[\Bar{x}\to \hat{y}]_n = 2\sqrt{n}+(y-x)n^{1/6}+n^{-1/6}\mathcal{S}_n(x,y)\,,
\]  
\noindent where \(\Bar{x} = 2xn^{-1/3}\), \(\hat{y} = 1+2yn^{-1/3}\) and \([x \to y]_n\) denotes last passage percolation from \((x,n)\) to \((y,1)\) over an ensemble of \(n\) independent two-sided Brownian motions. The proof in \cite{DOV} shows tightness and uniqueness of subsequential limits of the laws of \(\mathcal{S}_n(x,y)\); uniqueness of such limits is much more involved to show, since tightness results were already known for the Airy Sheet limit for certain models \cite{pr2018local}. 

\begin{theorem}[Airy Sheet Existence and Uniqueness]\label{thm: Airy Sheet}
	\label{T:intro-sheet}
	The Airy sheet exists and is unique in law. Moreover, for every \(n\), there exists a coupling so that
	\[
	B[(2x/n^{1/3},n) \to (1+2y/n^{1/3},1)] = 2\sqrt{n} + 2(y - x)n^{1/6} + n^{-1/6} (\mathcal{S}+o_n)(x,y),
	\]
	where \(o_n\) are random functions asymptotically small in the sense that on every compact set \(K\subset \R^2\) there exists \(a>1\) with
	\(\mathbb{E} a^{\sup_K |o_n|^{3/2}}\to 1\).
\end{theorem}

Noteworthy is how the \(3:2:1\), or KPZ scaling appears under rescaling of the Airy Sheet. Define the \textbf{Airy sheet of scale} \(s\) to be 
\[
\mathcal{S}_{s}(x,y) = s\mathcal{S}(x/s^2, y/s^2).
\]

\noindent Then, in Proposition 9.2 of \cite{DOV}, it is shown that if \(\mathcal{S}_s\) and \(\mathcal{S}_t\) are two independent Airy sheets of scales \(s\) and \(t\) respectively, then their metric composition, defined for any \((x,z)\in \R^2\) as 
\begin{equation}\label{eq: metric Airy}
    Q(x,z) = \sup_{y\in\R}\mathcal{S}_s(x,y)+\mathcal{S}_t(y,z)
\end{equation}

\noindent is an Airy sheet or scale \(r\), where \(r^3 = s^3+t^3\). In other words
\[
r^{-1}Q(xr^2,zr^2) = (s^3+t^3)^{-1/3}Q(xr^2,zr^2)
\]
\noindent is a standard Airy Sheet, whence the natural choice for the time : space : fluctuation scaling being \(3:2:1\) becomes apparent. This metric composition property is loosely reminiscent of how the variance of \(X_1 + X_2\) is \(\sigma_1^2+\sigma_2^2\), where \(X_1,X_2\) are independent, normally distributed with respective variances \(\sigma_1^2+\sigma_2^2\)\footnote{Observe the different exponents.}. 

\section{The Directed Landscape} 

The above suggests that metric composition can be viewed as a semigroup operation for the Airy Sheet, inspiring a definition of an analogous Brownian motion. For this, it is natural to have a parameter space internalising the spatial dependence of the Airy Sheets, being random functions, namely, the \textbf{directed $\R^4$},
$$
\mathbb R^4_\uparrow=\{(x,s;y,t)\in \R^4:s<t\}\,.
$$
Informally, $\R^4_\uparrow$ can be thought of as representing ordered pairs of points in spacetime in one dimension, where the coordinates $x$ and $y$ are spatial and the coordinates $s$ and $t$ are temporal. This leads to an important definition.

\begin{definition}\label{D:directed-landscape}The {\bf directed landscape} is a random continuous function $\mathcal{L}:\mathbb R^4_\uparrow \to \mathbb R$ satisfying the metric composition law
\begin{equation}
\label{E:dland-metric}
\mathcal{L}(x,r;y,t)=\max_{z\in \mathbb R} \mathcal{L}(x,r;z,s)+\mathcal{L}(z,s;y,t)
\end{equation}
for all \((x,r;y,t) \in  \mathbb R^4_\uparrow,\; s\in (r,t)\) and with the property that $\mathcal{L}(\cdot,t_i;\cdot,t_i+s_i^3)$
are independent Airy sheets of scale $s_i$ for any set of disjoint time intervals $(t_i,t_i+s_i^3)$.
\end{definition}

The construction of the Directed Landscape is the main result of the breakthrough DOV paper \cite{DOV}. We give a brief outline of the main steps. We focus on the first step where the comparison to the construction of Brownian motion is more apt.

\begin{enumerate}
    \item Construction on dense (dyadic times) subset of \(\R^4_\uparrow\)\label{enum: one}
    \item Extend by a.s. uniform continuity on compacts to conclude existence \label{enum: two}
    \item Obtain control of the growth of the directed landscape (shifted by a parabola) to show that it satisfies the metric composition law in full generality
    \label{enum: three}
\end{enumerate}

To make \ref{enum: one} more precise, let
\[
S_k = \{(x,t;y,s)\in \R^4: s<t, 2^kt, 2^ks\in\Z\},
\]
\noindent that is the set of points in the directed \(\R^4\) with dyadic time values of scale \(k\in\N\). It is clear that their union \(S =\bigcup_{k} S_k\) is dense in \(\R^4_\uparrow\) in the usual sense. The construction is inductive in nature and the idea is to define a collection of random functions 
\[
\mathcal{L}_k: D_k \coloneqq \{s<t: 2^kt, 2^ks\in\Z\} \longmapsto \mathcal{C}(\R^2; \R)
\]
and show that for \(k<k'\), the law of \(\mathcal{L}_k'\) restricted to \(D_k\) agrees with that of \(\mathcal{L}_k\). This consistency condition being met will then enable us to apply the classical Kolmogorov extension theorem to yield the existence of a process \(\mathcal{L}\) on 
\[
 D = \{s<t: 2^kt, 2^ks\in\Z \text{ for some } k\in \N\}.
\]
Finally, since all random functions \(D \mapsto \mathcal{C}(\R^2; \R)\) can be viewed as \(\R^S\)-valued random variables, this would finish the construction of the directed landscape on pairs of points in spacetime with dyadic time coordinates, which is dense in \(\R^4_\uparrow\) as aforementioned.\\

To define \(\mathcal{L}_k\) for \(k\in \N\), let \(B_i\) be independent Airy Sheets for \(i\in\Z\). With the definition of the directed landscape in mind, for \(t = i/2^k\) and \(s^3 = 1/2^k\) set
\[
\mathcal{L}_k(x,t;y,t+s^3) = sB_i(x/s^2, y/s^2)\,,
\]
that is, an Airy Sheet of scale \(s\). In words, for times in \(D_k\) with minimal spacing in time, one uses the above definition, being consistent with definition \ref{D:directed-landscape}. Now for times in \(D_k\) with non-minimal separation, the value of \(\mathcal{L}_k\) is that given by metric composition \ref{E:dland-metric}. More precisely, for \(j = i+\ell, \ell\geq 2\) set 
\begin{equation}\label{eq: directed comp dyadic}
\mathcal{L}_k(x,i/2^k;y, j/2^k) = \max_{(x_1,\dots, x_{\ell-1})\in\R^{\ell-1}}\displaystyle\sum_{q=1}^\ell \mathcal{L}_k\left(x_{q-1},\frac{i+q-1}{2^k};x_q, \frac{i+q}{2^k}\right)
\end{equation}
with \(x_0 = x, x_q = y\), that is an \(\ell\)-fold iteration of metric composition. Observe that for \(x,y\in \R\) for \(t = i/2^k\)
\[
\mathcal{L}_{k+1}\left(x,t;y, t+1/2^k\right) = \mathcal{L}_{k+1}\left(x,\frac{2i}{2^{k+1}};y, \frac{2i+2}{2^{k+1}}\right)\]
\[
= \max_{x_1\in\R}\left\{\displaystyle \mathcal{L}_{k+1}\left(x,\frac{2i}{2^{k+1}};x_1, \frac{2i+1}{2^{k+1}}\right)+\mathcal{L}_{k+1}\left(x_1,\frac{2i+1}{2^{k+1}};y, \frac{2i+2}{2^{k+1}}\right)\right\}
\]
which has the same law as 
\[
\max_{z\in\R}\displaystyle sB(x/s^2, z/s^2)+sB'(z/s^2, y/s^2)\,,
\]
where \(s = 1/2^{\frac{k+1}{3}}\) and \(B, B'\) are independent Airy Sheets. By the metric composition law for Airy Sheets, \ref{eq: directed comp dyadic} has the same law as an Airy Sheet of scale \(r\), where \(r^3 = 2s^3 = 1/2^k\). This turns out to have the same law as \(\mathcal{L}_{k}\left(x,t;y, t+1/2^k\right)\), by construction. This means that the consistency condition in Kolmogorov's extension theorem is satisfied. Also, the above readily implies that \(\mathcal{L}\) as constructed on \(S\) satisfies all the conditions in the definition of the directed landscape \ref{D:directed-landscape}. Uniqueness of the law of such a measure satisfying the definition \ref{D:directed-landscape} also follows from Kolmogorov's uniqueness theorem since the finite dimensional marginals are uniquely determined by the above.\\

Now, having obtained a random process \(\mathcal{L}:S\to \R\), a sufficient condition to show that it has a unique continuous extension to all of \(\R^4_\uparrow\), it suffices to show \ref{enum: two}, namely that \(\mathcal{L}\) is almost surely uniformly continuous on \(K\cap S\cap\Q^4\) for any compact \(K\subset \R^4_\uparrow\). The way the authors this in \cite{DOV} is by leveraging an explicit tail bound on two-point differences for 
\(\mathcal{L}\), which is the content of Lemma \(10.4\) therein. Briefly, to deal with \ref{enum: three}, Corollary 10.7 in \cite{DOV} gives the desired control in the growth rate of the parabolically shifted directed landscape to conclude that is satisfies the metric composition law.\\

Perhaps not too surprisingly by now, there is a connection between the Directed landscape and Brownian last passage percolation, which we phrase in the form of a theorem, Theorem 1.5 in \cite{DOV}. Letting $(x,s)_n=(s+2x/n^{1/3},-\lfloor sn\rfloor )$, the translation between limiting and pre-limiting locations, we have the following.

\begin{theorem}[Full scaling limit of Brownian last passage percolation]
\label{T:lp-limit}
There exists a coupling of Brownian last passage percolation and the directed landscape $\mathcal{L}$ so that
$$
B_n[(x,s)_n\to (y,t)_n ]\;\; = \;\;2(t-s)\sqrt{n} + 2(y - x)n^{1/6} + n^{-1/6} (\mathcal{L}+o_n)(x,s;y,t).
$$
Here each $B_n\in \mathcal{C}^{\Z}$ is an ensemble of independent two-sided Brownian motions. Each $o_n$ is a random function asymptotically small in the sense that on every compact set $K\subset \R^4_{\uparrow}$ there exists $a>1$ with
$\mathbb{E} a^{\sup_K |o_n|^{3/4}}\to 1$.
\end{theorem}

Having outlined the construction of the directed landscape and the Airy sheet, highlighting their important features and connections to Brownian last passage percolation, we now move onto discussion of the KPZ fixed point and how it can be described in terms of the directed landscape, bringing us closer to the proof of absolute continuity with respect to Brownian motion on compacts. More precisely, noting from \ref{D:directed-landscape}, that
\[
\mathcal{L}(x,0;y,t) = \mathcal{S}_{t^{1/3}}(x,y) = t^{1/3}\mathcal{S}(t^{-2/3}x, t^{-2/3}y), \quad t>0
\]
\noindent for an Airy Sheet \(\mathcal{S}\). Given an initial condition  \(h(0,\cdot)\), the KPZ fixed point of Matetski, Quastel and Remenik \cite{matetski2021kpz} can be expressed in terms of the directed landscape as
\begin{equation}\label{def: KPZ fixed point}
    h(t,y) = \sup_{x_0\in\R}\mathcal{L}(x,0;y,t), \quad \text{ for all } y\in \R.
\end{equation}

The construction of the directed landscape now allows for a variational description of the KPZ fixed point. This follows from \ref{eq: variation airy} in \cite{matetski2021kpz}, where the authors obtain a variational formula for the KPZ fixed point in terms of the Airy Sheet, now also shown to be unique in law, see \cite{DOV}. We now turn our attention to the main result of \cite{sarkar2021brownian}, namely, that the KPZ fixed point is absolutely continuous with respect to Brownian motion on compacts.

\chapter{Absolute continuity of Brownian last passage percolation}

We are moving towards proving the absolute continuity of Brownian last passage percolation with respect to Brownian motion. To do this, one needs to establish two preliminary results regarding the top line of the Pitman transform (\textit{melon}) \ref{Pitman} of two independent Brownian motions. To this end, we first prove the following lemma, which follows from the definition of the Pitman transform.\\

\begin{lemma}\label{lemma: Pitman transform is Markov}
    Let $(B_1,B_2)$ be a pair of two independent Brownian motions on the positive reals starting from $(0,b_1), (0,b_2)$ respectively. Then, the pair $X_t \coloneqq (WB_1(t), B_2(t))_{t\geq 0}$ is a Markov process in $t$ with respect to the filtration $(\mathcal{F}_t^{X})_{t\geq0} = (\sigma(X_s):0\leq s\leq t)_{t\geq0}$ generated by $(WB_1,B_2)$. 
\end{lemma}

\begin{proof}
    Let $f:\R^2\to\R$ be a bounded Borel-measurable function, fix $0\leq s<t$ and consider
    \[
    \mathbb{E}[f(WB_1(t),B_2(t))|\mathcal{F}_s^{X}]
    \]
    \[
    = \mathbb{E}\left[f\left(B_1(t)+\max\left(\max_{0\leq r\leq t}(B_2(r)-B_1(r)),0\right),B_2(t)\right)|\mathcal{F}_s^{X}\right]
    \]
    \begin{equation*}
    \begin{aligned}
    = \mathbb{E}\Bigl[f\Bigl(B_1(s)+\Delta_1B+\max\left(\max_{0\leq r\leq s}(B_2(r)-B_1(r)),
    \max_{s\leq r\leq t}(B_2(r)-B_1(r)),0\right),\\
    B_2(t)\Bigr)|\mathcal{F}_s^{X}\Bigr]
    \end{aligned}
    \end{equation*}
    where $\Delta_1B = B_1(t)-B_1(s)$ and $\Delta_2B = B_2(t)-B_2(s)$. Now,
    \begin{equation*}
    \begin{aligned}
        = \mathbb{E}\Bigl[f\Bigl(B_1(s)+\Delta_1B+\max\left(WB_1(s)-B_1(s),
        \max_{s\leq r\leq t}(B_2(r)-B_1(r))\right),\\
        B_2(s)+\Delta_2B\Bigr)|\mathcal{F}_s^{X}\Bigr]\,.
    \end{aligned}
    \end{equation*}
    Pushing the $B_1(s)$ term through the maximum gives
    \[
    = \mathbb{E}\left[f\left(\Delta_1B+\max\left(WB_1(s),
    M_s^t+B_2(s)\right), B_2(s)+\Delta_2B\right)|\mathcal{F}_s^{X}\right]\,,
    \]
    where $M_s^t = \displaystyle\max_{s\leq r\leq t}((B_2(r)-B_2(s))-(B_1(r)-B_1(s))), \Delta_1B, \Delta_2B \indep \mathcal{F}_s^{X}$ by the Markov property of Brownian motion. We thus conclude
    \[
     \mathbb{E}[f(WB_1(t),B_2(t))|\mathcal{F}_s^{X}] = g(WB_1(s), B_2(s)) = \mathbb{E}[f(WB_1(t),B_2(t))|X_s]\,,
    \]
    where $g:\R^2\to \R$ equals
    \[
    g(x,y) = \mathbb{E}\left[f\left(\Delta_1B+\max\left(x,
    M_s^t+y\right), y+\Delta_2B\right)\right]
    \]
    by independence, thereby concluding the proof that $(WB_1, B_2)$ is a Markov process. 
\end{proof}

Now, Lemma \ref{lemma: Pitman transform is Markov} combined with the general theory of Markov processes, yield the following lemma, Lemma $4.1$ in \cite{sarkar2021brownian}. 

\begin{lemma}\label{lemma: top line Pitman}
    Let $0<L<R$ and $b_1<b_2\in\R$. Let $B_1, B_2$ be two independent Brownian motions on the positive reals starting from $(0,b_1), (0,b_2)$ respectively. Let $WB_1(\cdot)$ be the top line of the Pitman transform \ref{Pitman} of the pair $(B_1, B_2)$. The, the law of $WB_1$ on the interval $[L,R]$ is absolutely continuous with respect to the law of a standard Brownian motion starting from the origin, again restricted to $[L,R]$.
\end{lemma}

Lemma \ref{lemma: top line Pitman} enables the proof of the absolute continuity of Brownian last passage percolation with respect to standard Brownian motion. In particular, the proof of the theorem below will rely on an induction argument, wherein the aforementioned lemma will establish the induction step as will become clear below.

\begin{theorem}[Absolute continuity of Brownian last passage percolation]\label{thm: abs cont brownian lpp}
    For any $y_0\in\R_+$ and $r\in(0,y_0)$, with $\mu_r$ denoting the law of a Brownian motion starting from the origin with diffusion parameter $2$ on the interval $[r,y_0]$. Let $\mathbf{\mathcal{B}} = (\mathcal{B}_1, \mathcal{B}_1, \cdots)\in\mathcal{C}^{\N}$ be an ensemble of independent standard Brownian motions with diffusion parameter $2$. Fix $k\in\N$ and $\mathbf{g}=(g_1,g_2,\cdots, g_k)\in\R^k$ and let $\xi_{k,\mathbf{g},r}$ denote the law
    \[
     \mathcal{H}_{k,\mathbf{g}}(y) \coloneqq \displaystyle\max_{1\leq \ell\leq k}(g_\ell+\mathbf{\mathcal{B}}[(0,\ell)\to(y,1)])\,.
    \]
    Then for all $k\in\N$, $\mathbf{g}=(g_1,g_2,\cdots, g_k)\in\R^k$ and $r\in(0,y_0)$,\footnote{$\lll$ \text{ denotes absolute continuity from here onwards.}}
    \[
    \xi_{k,\mathbf{g},r}\lll \mu_r\,.
    \]
\end{theorem}

\begin{proof}[Proof (Proceed via induction)]
$ $\newline

\noindent\underline{Case $k=1$:} Fix $g_1\in\R$ and $r\in(0,y_0)$. The claim follows immediately since 
\[
\mathcal{H}_{1,g_1}(y) = g_1+\mathbf{\mathcal{B}}[(0,1)\to(y,1)] = g_1+\mathcal{B}_1(y)\,,
\]
where $\mathcal{B}_1(y)$ is a Brownian motion with diffusion parameter $2$. 

\noindent\underline{Case $k=2$:} This essentially follows from lemma \ref{lemma: top line Pitman}. Indeed, from lemma \ref{lemma: Pitman melon} and that $\mathcal{B}_1(0) = \mathcal{B}_1(0)$, we observe that
\[
\mathcal{H}_{2,\mathbf{g}}(y) = \displaystyle\max_{1\leq \ell\leq 2}(g_\ell+\mathbf{\mathcal{B}}[(0,\ell)\to(y,1)]) = \displaystyle\max_{1\leq \ell\leq 2}(g_\ell+\mathbf{\hat{\mathcal{B}}}[(0,\ell)\to(y,1)])= W\hat{\mathcal{B}}_1(y)\,,
\]
where $\hat{\mathcal{B}} = (g_1+\mathcal{B}_1, g_2+\mathcal{B}_2)$, that is, an ensemble of two independent Brownian motions started from $(0,g_1)$ and $(0,g_2)$ with diffusion parameter $2$. Now applying Lemma \ref{lemma: top line Pitman} with $L = r, R = y_0$, we get $\xi_{2,\mathbf{g},r}\lll\mu_r$.

\noindent\underline{Case $k\geq 3$:} Here the idea is to exploit the associativity of the $\max(\cdot,\cdot)$ function and reduce the problem to the case for two lines. Suppose the claim holds for $k-1\in\N$, that is $\xi_{k-1, \mathbf{g},r}\lll\mu_r$ for all $\mathbf{g}\in\R^{k-1}$ and $r\in(0,y_0)$. We have
\[
\mathcal{H}_{k,\mathbf{g}}(y) = \max\Bigl(g_1+\mathcal{B}_1(y), \displaystyle\max_{2\leq \ell\leq k}\Bigl(g_\ell+\mathbf{\mathcal{B}}[(0,\ell)\to(y,1)]\Bigr)\Bigr)\,.
\]
Observe that the metric composition law \ref{Lemma: Metric Composition} gives 
\[
\mathbf{\mathcal{B}}[(0,\ell)\to(y,1)] = \displaystyle\sup_{0\leq t\leq y}(\mathbf{\mathcal{B}}[(0,\ell)\to(t,2)]+\mathcal{B}_1(y)-\mathcal{B}_1(t)).
\]
Now, we compute
\[
\displaystyle\max_{2\leq \ell\leq k}(g_\ell+\mathbf{\mathcal{B}}[(0,\ell)\to(y,1)])
\]
\[
= \displaystyle\max_{2\leq \ell\leq k}\Bigl(g_\ell+\displaystyle\sup_{0\leq t\leq y}\Bigl(\mathbf{\mathcal{B}}[(0,\ell)\to(t,2)]+\mathcal{B}_1(y)-\mathcal{B}_1(t)\Bigr)\Bigr)
\]
\[
= \displaystyle\sup_{0\leq t\leq y}\Bigr(\mathcal{H}_{k-1,\mathbf{g}'}'(t)+\mathcal{B}_1(y)-\mathcal{B}_1(t)\Bigl)\]
\[
= \mathcal{H}_{k-1,\mathbf{g}'}'(0)+\displaystyle\sup_{0\leq t\leq y}\Bigr(\mathcal{H}_{k-1,\mathbf{g}'}'(t)-\mathcal{H}_{k-1,\mathbf{g}'}'(0)+\mathcal{B}_1(y)-\mathcal{B}_1(t)\Bigl)\,,
\]
where $\mathbf{g}' = (g_2,\cdots, g_k)\in\R^{k-1}$ and $\mathcal{H}_{k-1,\mathbf{g}'}'(y) \coloneqq \displaystyle\max_{2\leq \ell\leq k}(g_\ell+\mathbf{\mathcal{B}}[(0,\ell)\to(y,2)])$. Notice that this corresponds to precisely last passage percolation on the two-line ensemble $\mathcal{L} = (\mathcal{B}_1, \mathcal{H}_{k-1,\mathbf{g}'}')$, that is
\[
\displaystyle\max_{2\leq \ell\leq k}(g_\ell+\mathbf{\mathcal{B}}[(0,\ell)\to(y,1)]) = \mathcal{H}_{k-1,\mathbf{g}'}'(0) + \mathcal{L}[(0,2)\to(y,1)].
\]
Hence, 
\[
\mathcal{H}_{k,\mathbf{g}}(y) = \max\Bigl(g_1+\mathcal{B}_1(y), \mathcal{H}_{k-1,\mathbf{g}'}'(0) + \mathcal{L}[(0,2)\to(y,1)]\Bigr)
\]
\[
= \max\Bigl(g_1+\mathcal{B}_1(y), \mathcal{H}_{k-1,\mathbf{g}'}'(0) + \mathcal{L}[(0,2)\to(y,1)]\Bigr)
\]
\[
= \max\Bigl(g_1+\mathcal{L}[(0,1)\to(y,1)], \mathcal{H}_{k-1,\mathbf{g}'}'(0) + \mathcal{L}[(0,2)\to(y,1)]\Bigr)
\]
\[
= W\hat{\mathcal{L}}_1(y)\,,
\]
where $\hat{\mathcal{L}} = (g_1+\mathcal{B}_1, \mathcal{H}_{k-1,\mathbf{g}'}')$, using the fact that $\mathcal{B}_1(y) = \mathbf{\mathcal{B}}[(0,1)\to(y,1)] = \mathcal{L}[(0,1)\to(y,1)]$ and Lemma \ref{lemma: Pitman melon}. In other words, $\mathcal{H}_{k,\mathbf{g}}$ is the top line of the melon of $\hat{\mathcal{L}}$. Since $\mathcal{B}_1 \indep \mathcal{H}_{k-1,\mathbf{g}'}'$ \footnote{From here onwards $\indep$ stands for the independence of sigma algebras (here generated by the two random variables).} as $\mathcal{H}_{k-1,\mathbf{g}'}'$ only depends on $\mathcal{B}_2,\cdots, \mathcal{B}_k$. Thus, by the induction hypothesis, $\xi_{k-1, \mathbf{g}',r/2} \lll \mu_{r/2}$ where the $\xi_{k-1, \mathbf{g}',r/2}$ is law of $\mathcal{H}_{k-1,\mathbf{g}'}'$ restricted to the interval $[r/2, y_0]$.\\

Note that the obstruction to directly applying Lemma \ref{lemma: top line Pitman} and concluding the proof is that the ensemble $\mathbf{\mathcal{L}}$ does not comprise of two independent Brownian motions. The second line is some continuous process, that nevertheless was just established to be absolutely continuous on the smaller interval $[r/2, y_0]$ with respect to the law of a standard Brownian motion starting from the origin restricted to $[r/2, y_0]$. However, with little more work, that is Corollary $4.2$ in \cite{sarkar2021brownian}, the above is sufficient to conclude that  
\[
\mathcal{L}(W\hat{\mathcal{L}}_1|_{[r/2, y_0]})\lll\mu_r\,,
\]
which concludes the proof.

\end{proof}

\chapter{The KPZ fixed point is Brownian on compacts}

The statement of the theorem requires unpacking and making precise some terms such as the type of initial condition \(h_0\) for which the KPZ fixed point is defined, and the precise meaning of the phrase \textit{Brownian on compacts}. To this end, we make two important definitions, following \cite{sarkar2021brownian}.\\

A function $f:\R\to \R\cup \{-\infty\}$ will be called a $t$-{\bf finitary initial condition} if $f(x)\neq -\infty$ for some $x$, $f$ is bounded from above on any compact interval and asymptotically
\begin{equation}\label{eq: t-finitary}
\frac{f(x)-x^2/t}{|x|}\to -\infty, \quad \text{ as } x\to \infty.
\end{equation}
A finitary initial condition (for some, or equivalently, all $t$) will be said to be {\bf compactly defined} if $f(x)=-\infty$ outside a compact set.\\

The name of this definition comes from the fact that for a $t_0$-finitary initial condition $h_0$, the KPZ fixed point, as defined in \ref{def: KPZ fixed point} is finite for all $y\in\R,t\in(0,t_0]$. This essentially follows from the growth estimate of the directed landscape $\mathcal{L}$, Corollary 10.7 in \cite{DOV}\\
\begin{equation}\label{eq: directed landscape growth estimate}
    |\mathcal{L}(x,0;y,t) + (x-y)^2/t| \leq C(1+|x|^{1/5}+|y|^{1/5}+|t|^{1/5})\theta(t)
\end{equation}
for $x,y\in\R$, where $C>0$ is a random (a.s. finite) constant and $\theta(t) = t^{1/3}\vee \log_+^{4/3} (1/t)$ is a positive continuous function. For $(y,t)\in(0,t_0]$, the bound becomes
\begin{equation}\label{eq: directed l bound}
|\mathcal{L}(x,0;y,t) + (x-y)^2/t| \leq C'(1+|t_0|^{1/5}) + C'|x|^{1/5}+ C'|y|^{1/5}\,,
\end{equation}
with $C' = C\sup_{t\in[0,t_0]}\theta(t)$. Thus, we obtain for $h_t(y) = \sup_{x\in\R}\{\mathcal{L}(x,0;y,t)+h_0(x)\}$
\[
 h_t(y)  \leq \sup_{x\in\R}\left\{\mathcal{L}(x,0;y,t)+\frac{(x-y)^2}{t}+h_0(x)-\frac{x^2}{t}+\frac{2xy}{t}-\frac{y^2}{t}\right\}
\]
and the bound \ref{eq: directed l bound} yields
\[
h_t(y)\leq C'((1+|t_0|^{1/5})+|y|^{1/5})-\frac{y^2}{t}+ \sup_{x\in\R}\left\{h_0(x)-\frac{x^2}{t_0}+\frac{2xy}{t}+C'|x|^{1/5}\right\} 
\]
for all $(y,t)\in\R\times(0,t_0]$. Notice that the right hand must always be finite, being the supremum of a function, bounded above on compacts which tends to $-\infty$ as $|x|\to \infty$, since $h_0$ was assumed $t_0$-finitary. This gives that almost surely $h_t(y)$ is finite for all $y\in\R,t\in(0,t_0]$, as claimed. One can do slightly more when restricted to compact sets. Indeed, for any compact $K\subset\R\times (0,t_0]$ and $(y,t)\in K$
\newpage
\[
h_t(y) \leq C'((1+|t_0|^{1/5})+\text{diam}(K)^{1/5}) + \sup_{x\in\R}\left\{h_0(x)-\frac{x^2}{t_0}+\frac{2|x|\text{diam}(K)}{\tau}+C'|x|^{1/5}\right\}\,, 
\]
where $M>0$ is a fixed positive number such that $K\subset B_M\times[\tau,T], B_M \coloneqq\{x\in\R^2:|x|\leq M\}$ for $0<\tau\leq T\leq t_0$, which exists by compactness. This bound on the right hand side is uniform in $K$ and converges to $-\infty$ as $|x|\to\infty$. One can also obtain in an almost identical fashion a uniform lower bound for $h_t(y)$ on $K$. Thus, for $(y,t)\in K$ the function
\[ 
x\longmapsto \mathcal{L}(x,0;y,t)+h_0(x)
\]
attains its supremum on a random interval $[L,R]$ that depends on $K$ and $C'$, leading to 
\begin{equation}\label{eq: compact containment}
h_t(y) = \sup_{x\in[L,R]}(\mathcal{L}(x,0;y,t)+h_0(x)), \quad (y,t)\in K.
\end{equation}

Examples of functions that satisfy condition \ref{eq: t-finitary} for all \(t>0\) are the initial conditions in the space of upper semicontinuous functions $UC$
\begin{equation}\label{eq: UC}
    \mathfrak{h}:\R \to [-\infty, \infty) \text{ with } \mathfrak{h}(x)\leq \alpha+\gamma|x| \text{ for finite } \alpha, \gamma \text{ and } \mathfrak{h}\not \equiv -\infty\,,
 \end{equation}
for which the KPZ fixed point was established in \cite{matetski2021kpz}. Additionally, any function $f\not \equiv -\infty$, satisfying $f(x)\leq C+x^2/t-|x|^{\beta}\log(1+|x|)$ for some $C>0$ and $\beta\in[1,\infty)$, is a $t$-finitary initial condition for any $t$ positive. This shows that this definition of admissible initial conditions can be quite  general, and in some sense, bearing the remark made earlier, the most general possible.\\

Now, to elucidate the meaning of \textit{locally Brownian}, we say that random function $F$ in $\mathcal{C}(\R;\R)$ is called {\bf Brownian on compacts} if for all $y_1<y_2$, the law of the $\mathcal{C}_0[y_1,y_2]$-valued random variable
$$y\mapsto F(y)-F(y_1)\,,$$
is absolutely continuous with respect to the law of a Brownian motion starting from $(y_1,0)$ with diffusion parameter $\sigma^2$ on $[y_1,y_2]$. Owing to convention, herein we use $\sigma^2=2$. We are now ready to state the main result of \cite{sarkar2021brownian}.

\begin{bigtheo}[Absolute continuity of KPZ fixed point on compacts]\label{thm: Absolute continuity KPZ} Let $t>0$; then for any $t$-finitary initial condition $h_0$ the random function
\[h_t(y) = \sup_{x\in \R} (h_0(x) + \mathcal{L}(x, 0; y, t)))\]
is Brownian on compacts.
\end{bigtheo}

\begin{remark}
    This form of local Brownainness, namely, being Brownian on compacts strictly supersedes that introduced in the paper by Quastel, Matetski and Remenik \cite{matetski2021kpz}, and leads to H\"{o}lder $\frac{1}{2}-$ spatial regularity of the KPZ fixed point, see corollaries \ref{c:wlocalbrown} and \ref{cor: Holder reg} respectively in the Appendix.
\end{remark}

The proof is elegant, owing to the framework developed thus far, combining all of the results gathered along the way. Notice that by scaling, it is enough to show Theorem \ref{thm: Absolute continuity KPZ} for $t = 1$, namely, to show that 
\[
h(y)\coloneqq \sup_{x\in \R} (h_0(x) + \mathcal{S}(x,y))
\]
is Brownian on compacts, where $\mathcal{S}(x,y)$ is an Airy Sheet since $\mathcal{L}(x, 0; y, t)$ has the law of an Airy Sheet of scale $t$. More precisely, notice that for a $t$-finitary initial condition $h_0$, 
\[
h_t(y) = \sup_{x\in\R}(\mathcal{L}(x,0;y,t)+h_0(x))\stackrel{d}{=} \sup_{x\in\R}(t^{1/3}\mathcal{S}(t^{-2/3}x,t^{-2/3}y)+h_0(x))
\]
\[
\stackrel{d}{=}  \sup_{x\in\R}t^{1/3}(\mathcal{S}(x,t^{-2/3}y)+t^{-1/3}h_0(t^{2/3}x)).
\]  
Now, since one can easily check that $t^{-1/3}h_0(t^{2/3}x)$ is $1$-finitary, if we could show that 
 \[
 \hat{h}_t(y) = t^{-1/3}h_t(t^{2/3}y) \stackrel{d}{=}  \sup_{x\in\R}(\mathcal{S}(x,y)+t^{-1/3}h_0(t^{2/3}x))
 \]
 is locally Brownian, then applying the inverse scaling to $\hat{h}_t(y)$ yields that $h_t(y) \stackrel{d}{=} t^{1/3}\hat{h}_t(t^{-2/3}y)$ is locally Brownian \footnote{This is essentially due to the fact that this 2:1 space:fluctuation scaling is precisely the one that keeps the law of a standard Brownian motion invariant.} We have thus reduced the general case to the case where $t = 1$.\\

To this end, we now prove the local Brownian nature of the fixed point at $t=1$, albeit with the $1$-finitary initial condition replaced with that of simplified initial data, namely, \textit{continuous} and \textit{compactly defined} data $h_0$ defined on some fixed interval $[x_1,x_2]$ and with $h_0(x)=-\infty$ for all $x\not\in[x_1,x_2]$. In other words, we would like the following. 

\begin{prop}[Compactly supported initial data]\label{prop: compactly supp data KPZ}
    With  \[
    h(y) = \sup_{x\in[x_1,x_2]}(h_0(x)+\mathcal{S}(x,y))\]
    and $h_0$ \textbf{continuous} and compactly defined, then $h$ is Brownian on compacts.
\end{prop}
\begin{remark}
    The content of the proof of proposition \ref{prop: compactly supp data KPZ} is broken into lemmas \ref{lemma: restriction} and \ref{lemma: Gibbs}, and everything is combined thereafter to form the main argument of the above proposition.
\end{remark}
We first collect the result that amounts to reducing what one has to show to prove local Brownianness.

\begin{lemma}\label{lemma: restriction}
Fix any $1<x_0,y_0$ and $\hat{h}_0$ compactly defined and continuous on $[1,x_0]$. Consider  
\[
\hat{h}(y) = \sup_{x\in[1,x_0]}(\hat{h}_0(x)+\mathcal{S}(x,y)).
\]
If the law of the law of the $\mathcal{C}_0[1,y_0]$-valued random variable
$$y\mapsto \hat{h}(y)-\hat{h}(1)$$
is absolutely continuous with respect to the law of a Brownian motion starting from $(1,0)$ with diffusion parameter $2$ on $[1,y_0]$, then $h$ as in Proposition \ref{prop: compactly supp data KPZ} is Brownian on compacts.
\end{lemma}

\begin{proof}
Fix any $y_1, y_2\in\R$ with $y_1<y_2$. Suppose first that $x_1\geq y_1$. Then, with the bi-measurable bijection $\Phi: \mathcal{C}([y_1,y_2];\R) \to \mathcal{C}([1, y_2-y_1+1];\R)$ 
\[g\mapsto (y\mapsto g(y+y_1-1))\,,\]
we compute for $y\in [1, y_2-y_1+1]$
\[
\Phi(h(y)) = \sup_{x\in[x_1,x_2]}(h_0(x)+\mathcal{S}(x,y+y_1-1))
\]
\[
= \sup_{x\in[x_1,x_2]}(h_0(x)+\mathcal{S}(x-y_1+1,y))\,
\]
using the translation invariance enjoyed by the Airy sheet. We thus have,
\[
\Phi(h(y)) = \sup_{x\in[x_1-y_1+1,x_2-y_1+1]}(h_0(x+y_1-1)+\mathcal{S}(x,y))
\]
\[
= \sup_{x\in[1,x_2-y_1+1]}(h_0(x+y_1-1)+\mathcal{S}(x,y))\,,
\]
where $h_0(\cdot+y_1-1)$ is compactly defined on $[1,x_2-y_1+1]$ and using the fact that $x_1\leq y_1$ to enlarge the domain of the optimisation. Hence, using Proposition \ref{prop: compactly supp data KPZ}, we conclude that the law of 
\[
\mathcal{L}\Bigl(y\mapsto \Phi(h)(y)-\Phi(h)(1), y\in [1, y_2-y_1+1]\Bigr)\lll \mu_{[1, y_2-y_1+1]} = \Phi_*\mu_{[y_1,y_2]}\,.
\]
Setting $h_0\equiv-\infty$ on a slightly larger interval, while losing continuity, does not obstruct us from applying the above proposition, since the proof carries along in essentially the same way (e.g. with $[r,x_0]$ instead of $[1,x_0]$ for any $r\in [1,x_0]$, where one has $h_0$ continuous). Furthermore, observe that restricting $\Phi: \mathcal{C}_0([y_1,y_2];\R) \to \mathcal{C}_0([1, y_2-y_1+1];\R)$ to the subspace of continuous functions starting from zero yields 
\[
\Phi_{*}\mathcal{L}\Bigl(y\mapsto h(y)-h(y_1), y\in [y_1, y_2]\Bigr) = \mathcal{L}\Bigl(y\mapsto \Phi(h)(y)-\Phi(h)(1), y\in [1, y_2-y_1+1]\Bigr)\,,
\]
thereby giving 
\[
\Phi_{*}\mathcal{L}\Bigl(y\mapsto h(y)-h(y_1), y\in [y_1, y_2]\Bigr) \lll \Phi_*\mu_{[y_1,y_2]}\,.
\]
Now push forward again with the inverse map $\Phi^{-1}$  using the composition property of pushforward operation on measures to deduce 
\[
\mathcal{L}\Bigl(y\mapsto h(y)-h(y_1), y\in [y_1, y_2]\Bigr)\lll \mu_{[y_1,y_2]}\,.
\]
To deal with the remaining case, namely that $x_1<y_1$, observe that the above argument shows absolute continuity on a larger interval, $[x_1, y_2]$. That is
\begin{equation}\label{eq: abs cont large interval}
\mathcal{L}\Bigl(y\mapsto h(y)-h(x_1), y\in [x_1, y_2]\Bigr)\lll \mu_{[x_1,y_2]}\,.
\end{equation}
Define the measurable map $\Xi: \mathcal{C}_0([x_1,y_2];\R) \to \mathcal{C}_0([y_1, y_2];\R)$ \[g\mapsto (y\mapsto g(y)-g(y_1)).\]
Now, let
\[
A = \{f\in \mathcal{C}_0[y_1, y_2]:f(t_1)\in B_1, \cdots, f(t_k)\in B_{k}\}
\]
with $B_1, \cdots, B_k$ Borel and $y_1\leq t_1<\cdots<t_k\leq y_2$, be a cylinder set in $\mathcal{C}_0[y_1,y_2]$. Then we compute 
\[
\Xi_*(\mu_{[x_1,y_2]})(A) = \mu_{[x_1,y_2]}(\Xi^{-1}(A))
\]
\[
= \mu_{[x_1,y_2]}\Bigl(\{f\in \mathcal{C}_0[x_1, y_2]:f(t_1)-f(y_1)\in B_1, \cdots, f(t_k)-f(y_1)\in B_{k}\}\Bigr) = \mu_{[y_1,y_2]}(A)\,,
\]
where the last equality follows from standard facts of Gaussian measures. Since $A$ was arbitrary, and cylindrical sets generate the Borel sigma algebra on $\mathcal{C}_0([y_1, y_2];\R)$, we thus have $\Xi_*(\mu_{[x_1,y_2]}) = \mu_{[y_1,y_2]}$. Finally, \ref{eq: abs cont large interval} implies 
\[
\mathcal{L}\Bigl(y\mapsto h(y)-h(y_1), y\in [y_1, y_2]\Bigr)
=\Xi_*\mathcal{L}\Bigl(y\mapsto h(y)-h(x_1), y\in [x_1, y_2]\Bigr)\]
\[
\lll \Xi_*\mu_{[x_1,y_2]} = \mu_{[y_1,y_2]}
\]
after pushing forward both laws under $\Xi$, which concludes the proof of absolute continuity of $y\mapsto h(y)-h(y_1)$ on $[y_1,y_2]$ with respect to $\mu_{[y_1,y_2]}$. Since, $y_1<y_2$ were arbitrary, we are done.

\end{proof}

\begin{proof}[Proof of Proposition \ref{prop: compactly supp data KPZ}]
    Fix $y_1<y_2$ and consider the restriction of $h$ as above to the interval $[y_1,y_2]$. The above lemma, (Lemma \ref{lemma: restriction}) shows that it suffices to show the claim when $x_1 = y_1 = 1$ and $x_2=x_0$ and $y_2 = y_0$ for some $x_0,y_0>1$. This essentially holds since the law of $\mathcal{S}$ is invariant under translations, and one can extend the intervals in consideration without affecting the absolute continuity.\\

    Using the continuity of $h_0$ and $\mathcal{S}(x,y)$ one only has to work on the rationals, namely
    \[
    h(y) = \sup_{x\in[1,x_0]\cap\Q}(h_0(x)+\mathcal{S}(x,y))\,.
    \]
    Now, from Lemma 3.10 in \cite{sarkar2021brownian}, there is a random constant $L_0$ such that almost surely for all $y\in[1,y_0]$
    \[
    h(y) = \sup_{x\in[1,x_0]\cap\Q}(h_0(x)+\max_{\ell\leq L_0}\mathcal{A}[x\to(0,\ell)]+\mathcal{A}[(0,\ell)\to (y,1)])
    \]
    where $\mathcal{A}$ is an Airy line ensemble that is coupled to the Airy Sheet $\mathcal {S}$ as in definition \ref{def: Airy Sheet} and 
    \begin{equation}\label{eq:Airy limit}
    \mathcal{A}[x\to(0,\ell)]\coloneqq 
    \begin{cases}
        \mathcal{S}(x,0) & \ell = 1\\
        \displaystyle\lim_{k\to\infty}\mathcal{A}[x_k\to (0,\ell)]-\mathcal{A}[x_k\to (0,1)]+\mathcal{S}(x,0) & \ell >1, x\in\Q^+
    \end{cases}
    \end{equation}
    with $x_k = (-\sqrt{k/2x},k), k\in\N$.\\
    
    \noindent The fact that this limit exists and is well defined is the crux of Theorem 3.7 in \cite{sarkar2021brownian} and uses geometric properties of geodesics in the Airy line ensemble. In some sense, the downward parabolic curvature of the Airy line ensemble $\mathbf{\mathcal{A}} = (\mathcal{A}_1, \mathcal{A}_2,\cdots)$ (see figure \ref{fig: Gibbs}), which follows from the observation that for $i\in \N$, the process $\mathcal{A}_i(\cdot)+(\cdot)^2$ is stationary (see \cite{sarkar2021brownian}), forces rightmost last passage paths (rightmost geodesics) to almost surely eventually intersect in the far left end of the plane (Lemma 3.4 in \cite{sarkar2021brownian}).\\
    
    Now, exchanging the sup and max gives
    \[
    h(y) = \max_{\ell\leq L_0}(G_{\ell}+\mathcal{A}[(0,\ell)\to (y,1)])\,,
    \]
    where
    \[
    G_\ell \coloneqq  \sup_{x\in[1,x_0]\cap\Q}(h_0(x)+\mathcal{A}[x\to(0,\ell)])\,.
    \]
    For any $\ell\in\N$, $G_{\ell}$ enjoys the following two properties, namely, that almost surely $G_{\ell}<\infty$ and that it is measurable with respect to the sigma algebra $\mathcal{F}_{-}\coloneqq \sigma(\{\mathcal{A}_i(x):x\leq0,i=1,2,\cdots\})$. This is the content of Lemmas $3.8$ and $3.9$ respectively. The latter property essentially follows from the definition \ref{eq:Airy limit}.\\

    For a moment assume that for all $k\in \N$, the laws $\xi_k$ of 
    \[
    \mathcal{H}_k(y) \coloneqq \max_{1\leq\ell\leq k}(G_{\ell}+\mathcal{A}[(0,\ell)\to (y,1)])
    \] 
    are absolutely continuous to $\mu$, defined to be the law of a standard Brownian motion with diffusion parameter $\sigma^2 = 2$ on $[1,y_0]$. Then for any Borel set in $\mathcal{C}[1,y_0]$ with $\mu(A) = 0$, we compute 
    \[
    \mathbb{P}(\{h(y)\in A\}) = \displaystyle \sum_{k=1}^\infty \mathbb{P}(\{h(y)\in A\}\cap\{L_0 = k\}) \]
    \[
    = \displaystyle \sum_{k=1}^\infty \mathbb{P}(\{\mathcal{H}_k(y)\in A\}\cap\{L_0 = k\}) = 0\,,
    \]
    since for all $k\in\N,\mathbb{P}(\{\mathcal{H}_k(y)\in A\}\cap\{L_0 = k\})\leq\xi_k(A) = 0$, by absolute continuity. Thereby showing that the law of $y\mapsto h(y)-h(1)$ is absolutely continuous with respect to the law of a standard Brownian motion with diffusion parameter $2$ started from $(1,0)$ and restricted to $[1,y_0]$\footnote{This deduction follows from a very similar argument as in the end of Lemma \ref{lemma: restriction}, using $\Psi: \mathcal{C}([1,y_0];\R) \to \mathcal{C}_0([1,y_0];\R)$ the measurable map $g\mapsto (y\mapsto g(y)-g(1))$ instead of $\Xi$.}, which would conclude the proof.\\
    
    The fact that $\xi_k\lll\mu$, (where $\mu\lll\nu$ denotes absolute continuity of $\mu$ with respect to $\nu$) for all $k\in\N$ is the content of lemma \ref{lemma: Gibbs} which follows from Theorem \ref{thm: abs cont brownian lpp} and crucially uses the Brownian Gibbs property of the Airy line ensemble.
\end{proof}

\begin{lemma}\label{lemma: Gibbs}
For $k\in\N$, some $y_0>1$, $\mathcal{H}_k(y) \coloneqq \displaystyle\max_{1\leq\ell\leq k}(G_{\ell}+\mathcal{A}[(0,\ell)\to (y,1)])$ and $G_\ell \coloneqq (h_0(x)+\mathcal{A}[x\to(0,\ell)])$, as in Proposition \ref{prop: compactly supp data KPZ}, it follows that \[
\xi_k\lll\mu,\]
where $\xi_k$ is the law of $\mathcal{H}_k(y)$ on $[1,y_0]$ and $\mu$ is defined to be the law of a standard Brownian motion with diffusion parameter $\sigma^2 = 2$ on $[1,y_0]$.
\end{lemma}

\begin{proof}[Proof]
Let $\mathcal{F}_k$ denote the sigma algebra generated by
\[
\{\mathcal{A}_i(x): (i,x)\not\in\{1,2,\cdots,k\}\times(0,y_0+1)\}\,.
\]
Unravelling the Brownian Gibbs Property, \ref{thm: Brownian Gibbs}, we have that conditioning on $\mathcal{F}_k$, the law of the top $k$ lines of the Airy line ensemble between the interval $[1,y_0+1]$ is that of $k$ independent Brownian bridges $B_1, B_2, \cdots, B_k$ with diffusion parameter $2$ on $[0,y_0+1]$ conditioned not to intersect each other and the bottom line $\mathcal{A}_{k+1}$, starting and ending at $\mathcal{A}_i(0)$ and $\mathcal{A}_i(y_0+1)$ for $i = 1,2,\cdots, k$ respectively.\\

More precisely, with $k\in \N$ as above and $\bar{x},\bar{y}\in \R^k$ with $x_1>x_2>\cdots>x_k$ and $y_1>y_2>\dots>y_k$. Write $\mathcal{W}_{\bar{x},\bar{y}}$ for the law of $k$ independent Brownian bridges $B_i:[0,y_0+1] \to \R$, $1 \leq i \leq k$, that satisfy $B_i(0) = x_i$ and $B_i(y_0+1) = y_i$. Observe that for $f:[0,y_0+1] \to \R$ a measurable function such that $x_{k}>f(0)$ and $y_k>f(y_0+1)$, the \textit{non-crossing} event on $ [0,y_0+1]$ given by
\begin{equation*}
\mathcal{N}^f_{[0,y_0+1]} =\Big\{\text{for all } r\in [0,y_0+1],  B_{i}(r) > B_{j}(r) \textrm{ for all } 1\leq i<j\leq k \textrm{ and  $B_k(r) > f(r)$} \Big\}
\end{equation*}
has positive probability, which follows from standard facts about Brownian Bridges, see section 2.2.2 in \cite{corwin2014brownian}. In other words, the {\it acceptance probability} $\mathfrak{a}(\bar{x},\bar{y},f)= \mathcal{W}_{\bar{x},\bar{y}}(\mathcal{N}_{[0,y_0+1]}^f)$ is positive. We are thus able to define the conditional measure \[
\mathcal{W}_{\bar{x},\bar{y}} \big( \cdot \big\vert \mathcal{N}^f_{[0,y_0+1]} \big)
= \frac{\mathcal{W}_{\bar{x},\bar{y}}(\cdot\cap \mathcal{N}_{[0,y_0+1]}^f)}{\mathfrak{a}(\bar{x},\bar{y},f)}\]
denoted by $\mathfrak{B}_{\bar{x},\bar{y}}^f(\cdot)$. Now, from \cite{corwin2014brownian} and \cite{hammond2021brownian}, the Brownian Gibbs property as stated in Theorem \ref{thm: Brownian Gibbs} can be written in the more precise terms
\[
\text{Law}\Bigl(\mathbf{\mathcal{A}}|_{\{1,2,\cdots,k\}\times(0,y_0+1)} \text{ conditioned on } \mathcal{F}_k\Bigr) = \mathfrak{B}_{\bar{x},\bar{y}}^f
\]
with entry data $\bar{x} = (\mathcal{A}_i(0))_{1\leq i\leq k}$, $\bar{y} = (\mathcal{A}_i(y_0+1))_{1\leq i\leq k}$ and $f = \mathcal{A}_{k+1}$ which are all $\mathcal{F}_k$-measurable. Notice that for fixed data $\bar{x}, \bar{y}, f$ the measure $\mathfrak{B}_{\bar{x},\bar{y}}^f$ is absolutely continuous with respect to $\mathcal{W}_{\bar{x},\bar{y}}$, that is the law of $k$ independent Brownian bridges on $[0,y_0+1]$ starting at $(0,x_i)$ and ending at $(y_0+1,y_i)$ respectively, for $1\leq i\leq k$.\\

Observe that the choice of interval where the Brownian Gibbs property was applied was made so as to strictly contain $[0,y_0]$ to the right. This implies that given $\mathcal{F}_k$, the endpoints $\mathcal{A}_i(y_0), i = 1,2,\cdots, k$ have densities with respect to the Lebesgue measure, since the Brownian Gibbs property gives absolute continuity with respect to the aforementioned Brownian Bridge measure on the augmented interval $[0,y_0+1]$. The argument in Proposition 4.1 from \cite{corwin2014brownian} now implies that the law of the top $k$ lines of the Airy line ensemble between $[0,y_0]$ given $\mathcal{F}_k$ is absolutely continuous with respect to that of $k$ independent Brownian motions with diffusion parameter $2$ starting from $\mathcal{A}_1(0)>\mathcal{A}_2(0)>\cdots >\mathcal{A}_k(0)$.\\

Hence, using disintegration (Theorem \ref{thm: Disintegration} in the Appendix), we see that the conditional distribution of $\mathcal{A}_i(\cdot)-\mathcal{A}_i(0):[0,y_0]\mapsto\R$ for $i=1,2,\cdots,k$ given $\mathcal{F}_k$ is absolutely continuous with respect to $k$ independent standard Brownian motions $B_1,B_2,\cdots, B_k$ with diffusion parameter $2$.\\

To conclude, condition on the sigma algebra $\mathcal{F}_k$ and use again disintegration, noting that $G_\ell$ is $\mathcal{F}_{-}\subseteq\mathcal{F}_k$-measurable, Theorem \ref{thm: abs cont brownian lpp}, and the fact that last passage percolation over the Airy ensemble is a functional of the centered process $\mathcal{A}(\cdot)-\mathcal{A}(0)$.

\end{proof}

Now that Proposition \ref{prop: compactly supp data KPZ} has been established, one can generalise it and obtain the following proposition.
\begin{prop}[1-finitary data]\label{prop: Brownian 1-finitary}For any \textbf{continuous} $1$-finitary initial condition $h_0$, the random function 
\[
h(y) = \sup_{x\in\R}(h_0(x)+\mathcal{S}(x,y))
\]
is Brownian on compacts.
\end{prop}

\begin{proof}
Fix $y_1<y_2$ and let $L,R$ be as in \ref{eq: compact containment}, since we are working on the compact set $[y_1, y_2]\times\{1\}\subset \R\times(0,1]$. Thus
\[
h(y) = \sup_{x\in[L,R]}(h_0(x)+\mathcal{S}(x,y))\,.
\]
Now, for $n\in \N$, define the continuous compactly contained data 
\[
h_0^n(x) = \begin{cases} h(x) & x\in [-n,n]\\
-\infty& \text{otherwise}
\end{cases}
\]
and set
\[
\mathcal{H}_n(y) \coloneqq \sup_{x\in[-n,n]}(h_0^n(x)+\mathcal{S}(x,y))\,.
\]
Furthermore, let $\xi_n$ be the law of $\mathcal{H}_n(y)- \mathcal{H}_n(y_1)$ on $[y_1, y_2]$. Also, let $\mu$ be the law of a standard Brownian motion with diffusion parameter $\sigma^2 = 2$ starting from $(y_1, 0)$ on $[y_1,y_2]$. By Proposition \ref{prop: compactly supp data KPZ}, $\xi_n\lll\mu$ for all $n\in\N$. Thus for any measurable $A\in \mathcal{C}_0[y_1,y_2]$, a simple computation gives
\[
\mathbb{P}(\{h(y)-h(y_1)\in A\}) = \displaystyle \sum_{n=1}^\infty \mathbb{P}(\{h(y)-h(y_1)\in A\}\cap\{[L,R]\subseteq[-n,n]\})\]
\[
= \displaystyle \sum_{n=1}^\infty \mathbb{P}(\{\mathcal{H}_n(y)-\mathcal{H}_n(y_1)\in A\}\cap\{[L,R]\subseteq[-n,n]\}) \leq \displaystyle \sum_{n=1}^\infty \xi_n(A)  =  0\,,
\]
hence establishing the desired absolute continuity.
\end{proof}

Combining all of the above, we now are ready to embark on the proof of Theorem \ref{thm: Absolute continuity KPZ} proper. See flowchart \ref{fig:flowchart}, for a summary of the key steps of the proof.

\begin{proof}[Proof (Theorem \ref{thm: Absolute continuity KPZ})]
    
    The above considerations, in particular scaling arguments for the directed landscape allow us to reduce the proof of the theorem to the case $t=1$. Now, first suppose that the initial data $h_0(\cdot)$ is $1$-finitary and \textit{compactly defined}. That is there exist  $x_1<x_2$ such that $h_0$ is bounded above on the interval $[x_1,x_2]$ and $h_0(x)=-\infty$ for all $x\not\in[x_1,x_2]$. We now need to show that 
    \[
    h(y) = \sup_{x\in\R}(h_0(x)+\mathcal{S}(x,y)) \stackrel{d}{=} \sup_{x\in\R}(h_0(x)+\mathcal{L}(x,0;y,1))
    \]
    is Brownian on compacts. By the metric composition law for the directed landscape \ref{E:dland-metric},\newline 
    \[
    \mathcal{L}(x,0;y,1) = \displaystyle\sup_{z\in\R}(\mathcal{L}(x,0;z,1/2)+\mathcal{L}(z,1/2;y,1)).
    \]
    Some algebraic manipulation gives
    \[
    h(y) = \sup_{x\in\R}\Bigl(h_0(x)+\displaystyle\sup_{z\in\R}\Bigl(\mathcal{L}(x,0;z,1/2)+\mathcal{L}(z,1/2;y,1)\Bigr)\Bigr) 
    \]
    \[
    = \displaystyle\sup_{z\in\R}\Bigl(\mathcal{U}(z)+\mathcal{L}(z,1/2;y,1)\Bigr)\,,
    \]
    where $\mathcal{U}(z) = \displaystyle\sup_{x\in\R}(h_0(x)+\mathcal{L}(x,0;z,1/2))$, $z\in\R$. Note that the independent increment property of the directed landscape yields the independence of $\mathcal{U}(\cdot)$ and $\mathcal{L}(\cdot,1/2;y,1)$. Furthermore, a pointwise bound on difference in $z$ of $h_0(x)+\mathcal{L}(x,0;z,1/2)$, $x\in [x_1, x_2]$ yields the bound
    \[
    |\mathcal{U}(z_1)-\mathcal{U}(z_2)|\leq \displaystyle\sup_{x\in[x_1,x_2]}|\mathcal{L}(x,0;z_1,1/2)-\mathcal{L}(x,0;z_2,1/2)|.
    \]
    This establishes that $\mathcal{U}$ is a continuous process since the optimisation happens over a compact interval and the directed landscape itself is a $\mathcal{C}(\R^4_{\uparrow};\R)$-valued process. We now bound $\mathcal{U}$ from above. To this end, the growth estimate \ref{eq: directed landscape growth estimate} of the directed landscape gives for $x\in[x_1, x_2]$
    \[
    |\mathcal{L}(x,0;z,1/2) + 2(x-z)^2| \leq C(1+|x|^{1/5}+|z|^{1/5}+|1/2|^{1/5})\theta(1/2)
    \]
    which implies
    \[
    \mathcal{L}(x,0;z,1/2) \leq C(1+|x|^{1/5}+|z|^{1/5}+|1/2|^{1/5}) - 2(x-z)^2
    \]
    \[
    \leq C(|z|^{1/5}+1)\,,
    \]
    where $C>0$ is a random constant that changed from line to line. Thus,
    \[
    \mathcal{U}(z)\leq \displaystyle\sup_{x\in[x_1,x_2]}h_0(x)+\displaystyle\sup_{x\in[x_1,x_2]}\mathcal{L}(x,0;z,1/2)
    \leq C(|z|^{1/5}+1) 
    \]
    for some random constant $C>0$. Hence,
    \[
    \frac{\mathcal{U}(z)-2z^2}{|z|} \leq \frac{C(|z|^{1/5}+1)-2z^2}{|z|}\to-\infty, \quad |z|\to\infty.
    \]
    Thus, $\mathcal{U}(z)$ is almost surely $1/2$-finitary. Recall from the scaling property of the directed landscape in \ref{D:directed-landscape}, we have 
    \[
    \mathcal{L}(z,1/2;y,1) \stackrel{d}{=} 2^{-1/3}\mathcal{S}(2^{2/3}x, 2^{2/3}y)\,.
    \]
    Hence, we have 
    \[
    2^{1/3}h(2^{-2/3}y) \stackrel{d}{=}\sup_{z\in\R}\Bigl(\mathcal{S}(z,y)+2^{1/3}\mathcal{U}(2^{-2/3}z)\Bigr).
    \]
    Since $\hat{\mathcal{U}}(\cdot) = 2^{1/3}\mathcal{U}(2^{-2/3}\cdot)$ is continuous and almost surely $1$-finitary, Proposition \ref{prop: Brownian 1-finitary} gives that $h$ is locally Brownian after conditioning on $\mathcal{L}(\cdot,0;\cdot,1/2)$ and using independence, see Lemma \ref{lemma:cond independence}.\\

    Finally, for an arbitrary $1$-finitary initial condition $h_0$, we have previously seen that equation \ref{eq: compact containment} is satisfied, and so the proof proceeds by essentially repeating the argument in Proposition \ref{prop: Brownian 1-finitary}. 

\end{proof}

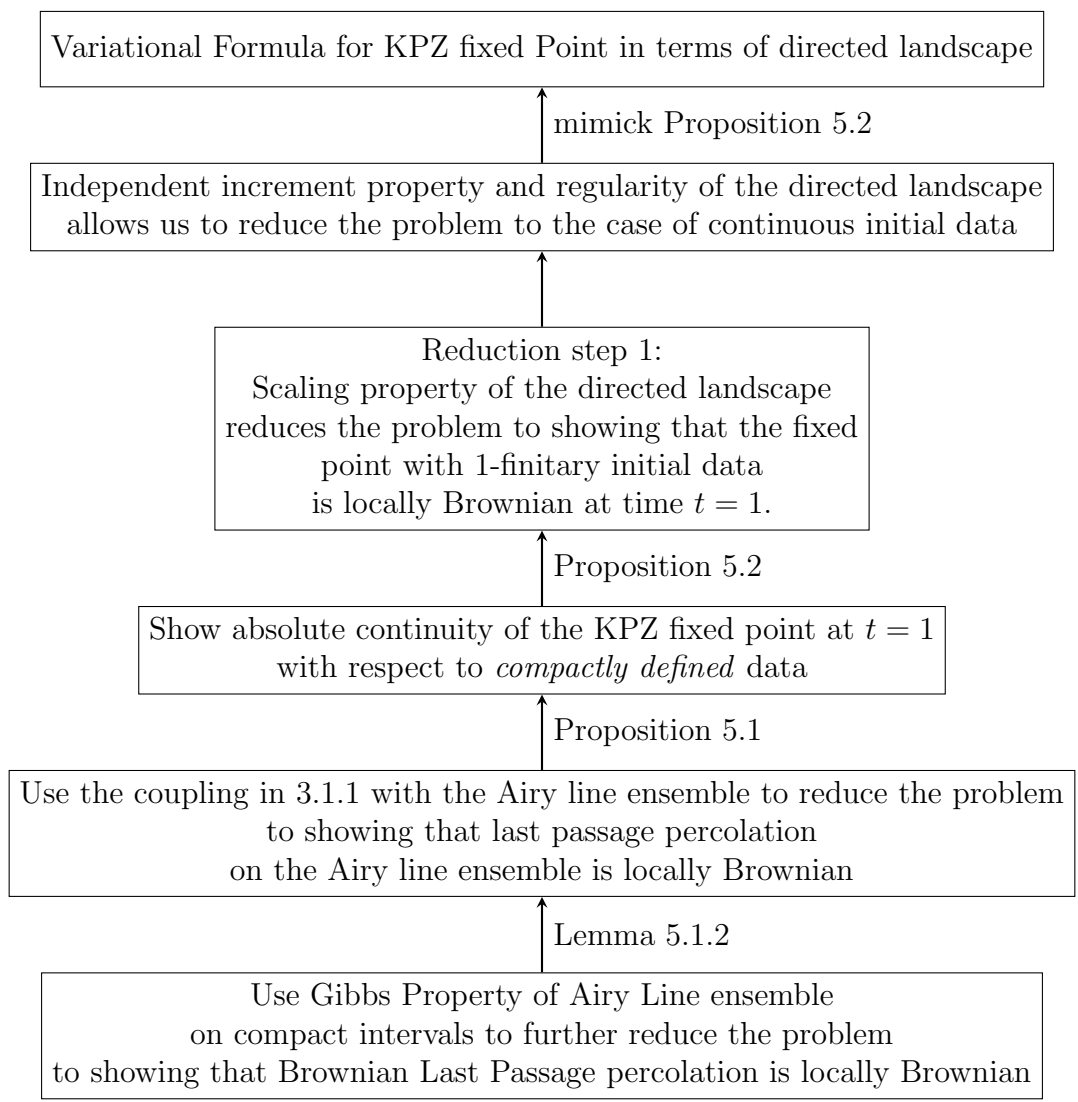
\begin{figure}[H]
  \vspace{1cm}
  \begin{center}
  \begin{tikzpicture}[
    node distance=1.5cm,
    startstop/.style={rectangle, minimum width=3cm, minimum height=1cm, text centered, draw=black},
    process/.style={rectangle, minimum width=3cm, minimum height=1cm, text centered, draw=black, align=center},
    decision/.style={rectangle, minimum width=3cm, minimum height=1cm, text centered, draw=black},
    arrow/.style={thick,->,>=stealth}
  ]

  \node (start) [startstop] {Variational Formula for KPZ fixed Point in terms of directed landscape};
\node (det) [process, below=1cm of start] {Independent increment property and regularity of the directed landscape\\
allows us to reduce the problem to the case of continuous initial data };
  \node (reduction1) [process, below=1cm of det] {Reduction step 1: \\
    Scaling property of the directed landscape\\
    reduces the problem to showing that the fixed\\
    point with 1-finitary initial data\\
    is locally Brownian at time $t=1$.};
    
  \node (reduction2) [process, below=1cm of reduction1] {Show absolute continuity of the KPZ fixed point at $t=1$\\
  with respect to \textit{compactly defined} data};
  \node (coupling) [process, below=1cm of reduction2] {Use the coupling in \ref{def: Airy Sheet} with the Airy line ensemble to reduce the problem\\ to showing that last passage percolation\\
  on the Airy line ensemble is locally Brownian};
  \node (gibbs) [process, below=1cm of coupling] {Use Gibbs Property of Airy Line ensemble\\
  on compact intervals to further reduce the problem\\
  to showing that Brownian Last Passage percolation is locally Brownian};

  \draw [arrow] (det) -- node[anchor=west] {mimick Proposition \ref{prop: Brownian 1-finitary}} (start);
  \draw [arrow] (reduction1) -- (det);
  \draw [arrow] (reduction2) -- node[anchor=west] {Proposition \ref{prop: Brownian 1-finitary}} (reduction1);
  \draw [arrow] (coupling)-- node[anchor=west] {Proposition \ref{prop: compactly supp data KPZ}}(reduction2) ;
  \draw [arrow]  (gibbs) -- node[anchor=west] {Lemma \ref{lemma: Gibbs}}(coupling);

  \coordinate[right=1cm of coupling] (dummy);
\end{tikzpicture}
\end{center}
  \caption{Flowchart of main steps in the proof of Theorem \ref{thm: Absolute continuity KPZ}.}
  \label{fig:flowchart}
\end{figure}

\chapter{Conclusion and future directions}

The proof of the locally Brownian nature of the KPZ fixed point under $t$-finitary initial conditions, the main result in \cite{sarkar2021brownian}, is the culmination of years of work spent trying to understand the probabilistic nature of the KPZ fixed point, conjectured to be the universal scaling limit, under the 3:2:1 scaling limit of a large class of random growth processes.\\

There are multiple directions for further exploration. There are several open problems pertaining to the geometry of random geodesics in the directed landscape, which all can be found in \cite{DOV}. These are not going to be touched upon, since they were not the direct focus of the essay. There are however, unanswered questions regarding the regularity of the Radon Nikodym derivative of the KPZ fixed point on compacts which we discuss below.\\

More precisely, one would like to know for which $p\in (1,\infty)$ the Radon Nikodym derivative is in $L^p$. This would be a desirably property to have since it would quantitatively strengthen the relationship between low-probability events of Brownian motion and that of the KPS fixed point \cite{calvert2019brownian}. More generally, in the setting of two finite measures $\mu\lll\nu$, one wants if possible to quantify the relationship between the $\delta>0$ and $\epsilon>0$ so that the implication $\nu(A)<\delta$ guarantees $\mu(A)<\epsilon$ for all measurable $A$\footnote{Recall the definition of absolute continuity of measures $\mu$ with respect to $\nu$, namely, that for all $\epsilon>0$, there exists $\delta>0$ such that for all $A$ measurable, if $\nu(A)<\delta$, then $\mu(A)<\epsilon$.}. This can be achieved if for instance one imposes that the Radon-Nikodym derivative $d\mu/d\nu\in L^p$, for some $p>1$. Then, for $A$ measurable, 
\[
\mu(A) = \displaystyle \int_{A}\frac{d\mu}{d\nu}d\nu
\leq \Bigl(\int_{A}\Bigl(\frac{d\mu}{d\nu}\Bigr)^pd\nu\Bigr)^{\frac{1}{p}}(\nu(A))^{\frac{p-1}{p}}
= \left\lVert\frac{d\mu}{d\nu}\right\rVert_{L^p(\nu)}(\nu(A))^{1-\frac{1}{p}}
\]
by applying H\"{o}lder's inequality. Thus, taking 
$\nu(A)<\delta$ gives
\[
\mu(A) \leq \left\lVert\frac{d\mu}{d\nu}\right\rVert_{L^p(\nu)}\delta^{1-\frac{1}{p}}
= \delta^{1-\frac{1}{p}+\frac{\log(\|d\mu/d\nu\|_{L^p(\nu)})}{\log(\delta)}} = \delta^{1-\frac{1}{p}-o(1)}
\]
as $\delta \searrow 0$. Hence, if the Radon-Nikodym  derivative was in all $L^p(\nu)$ for $p>1$, then taking $p\to\infty$ one would arrive at the bound $\mu(A)<\delta^{1-o(1)}$. Note that the case $p=1$ follows directly from absolute continuity and the finiteness of all measures involved and does not give a non-trivial estimate of the above form.\\

Some progress had already been made for the \textit{narrow wedge at $u$} initial data \[
h_0^u(x) = \begin{cases}
    0 & x = u\\
    -\infty & x\neq u
\end{cases}
\]
in \ref{eq: UC} in \cite{corwin2014brownian} it was shown that the top line of the Airy line ensemble, $\mathcal{A}_1$, or otherwise known as the Airy-2 process, is Brownian on compacts using the \textit{Brownian Gibbs property}. Recently, in \cite{calvert2019brownian} the Radon Nikodym derivative was shown to be in $L^p$ for all $p\in (1,\infty)$.\\

However, for general initial conditions, there is no such absolute continuity result to the present date. A form of Brownian regularity, called \textit{patchwork quilt of Brownian fabrics} was established in \cite{hammond2019patchwork} and \cite{calvert2019brownian}, the content of which states roughly that the KPZ fixed point $h(\cdot)$ on a unit interval is divided into a random number of subintervals, the patches, with random boundary points in such a way that the restriction of the profile to each patch is absolutely continuous with respect to a Brownian motion with Radon-Nikodym derivative in $L^p$ for all $p<3$. The content of conjecture $1.3$ in \cite{hammond2019patchwork} it to establish this with a single patch and show that the Radon-Nikodym derivative is in $L^p$ for all $p>0$. That the random number of such patches has polynomial tail bounds was shown in Theorem 6.10 in \cite{calvert2019brownian}, clarifying some aspects of the problem, though the above conjecture still remains open.\\

Theorem \ref{thm: Absolute continuity KPZ} amounts to establishing absolute continuity with a \textit{single} patch, see equation \ref{eq: compact containment}. The proof of Theorem \ref{thm: Absolute continuity KPZ} proceeds along very different lines than that in \cite{hammond2019patchwork} and \cite{calvert2019brownian}, save for the use of the Brownian Gibbs property. This suggests novel ideas will be needed to make progress.\\

Finally, we make some preliminary observations and point to parts of the proof of Theorem \ref{thm: Absolute continuity KPZ}, which should be related to improving the regularity of the KPZ fixed point. Observe that in the proof of Proposition \ref{prop: compactly supp data KPZ} for compactly supported data $h_0$, the KPZ fixed point has the expression 
\[
h(y) = \max_{\ell\leq L_0}(G_{\ell}+\mathcal{A}[(0,\ell)\to (y,1)]) = \max_{\ell\leq L_0}\mathcal{H}_\ell(y)
\]
with $\mathcal{H}_k(y) = \displaystyle\max_{\ell\leq k}(G_{\ell}+\mathcal{A}[(0,\ell)\to (y,1)])$, $L_0 = \pi[x'\to y'](0)$ and $x'\geq x_0,y'\geq y_0\in\Q_+$. See figure \ref{fig: Airy geodesic} for a graphical representation of $L_0$. One has that for all $n\in\N$, $\xi_n\lll\mu$, where $\xi_n$ is the law of $\mathcal{H}_k(y)$ restricted to $[1,y_0]$. It thus appears plausible to prove statements about regularity of the Radon Nikodym derivative by analysing the behaviour of $L_0$, in particular by controlling its tails. Similarly, a refined analysis of the randomness in the random interval $[L,R]$, appearing in the the more general case of $1$-finitary initial data in Proposition \ref{prop: Brownian 1-finitary} could also lead to further insights. Observe that aside from the initial condition $h_0$,the randomness depends on the deterministic size of the compact set under consideration and the random constant $C>0$ in Corollary $10.7$ in \cite{DOV}, which has exponential tails. 

\begin{figure}[H]
\vspace{4cm}
\includegraphics[width= 400pt]{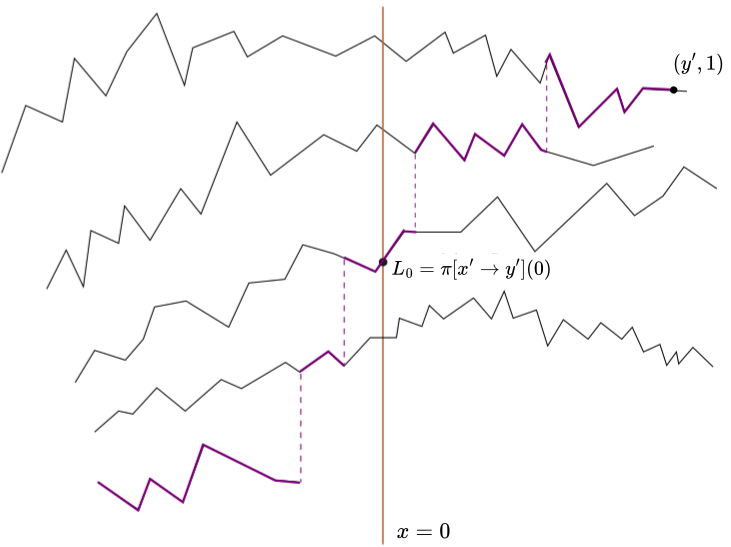}
\centering
\caption{Above is displayed the point $(0,L_0)$ at which the last passage path $\pi[x'\to y']$ on the Airy line ensemble $\mathcal{A} = (\mathcal{A}_1, \mathcal{A}_2,\cdots)$ (\color{purple} purple\color{black}) meets with the axis $\{x=0\}$, where $y'>1$. Here $L_0 = 3$ and the first four lines of $\mathcal{A}$ are shown. The last passage path $\pi[x'\to y']$ is defined in Definition 3.3 in \cite{sarkar2021brownian} (it essentially is a pointwise limit of geodesics on the Airy line ensemble with right endpoint fixed and left endpoint going \textquotedblleft down and to the left \textquotedblright on the ensemble).}
\label{fig: Airy geodesic}
\end{figure}

\appendix

\chapter{Measure theoretic results}

Below is quoted for convenience Theorem $8.5$ from \cite{kallenberg1997foundations}, which is a key result in modern probability theory and is frequently used when manipulating conditional expectations. First we make the definition of a \textit{regular conditional distribution} of a random element $\xi$ on some probability space $(\Omega, \mathcal{A}, \mathbb{P})$ taking values in a measurable space $(S,\mathcal{S})$.\\

\begin{definition}[Regular conditional distribution]
    With $\xi$, $\Omega, S$ as above and a sub-sigma algebra $\mathcal{F}\subset\mathcal{A}$. Then, by a \textit{regular conditional distribution} of $\xi$, we mean that there exists a function $\mu:\Omega\times\mathcal{S}\to \R_{+}$ such that $\mu(\omega,B)$
    \begin{itemize}
        \item is $\mathcal{F}$-measurable in $\omega\in \Omega$ for all $B\in\mathcal{S}$
        \item is a measure in $B\in\mathcal{S}$ for each $\omega\in \Omega$
        \item $\mu(\omega,S)=1$ for all $\omega\in \Omega$
    \end{itemize}
    and for all $B\in S$, $\mathbb{P}[\xi\in B|\mathcal{F}](\cdot)=\mu(\cdot, B)$ $\mathbb{P}$-a.s.
\end{definition}

\begin{theorem}\label{thm: Disintegration}
    Let $(\Omega, \mathcal{A}, \mathbb{P})$ be a probability space and $\xi$ and $\eta$ be random elements in two measurable spaces $(S,\mathcal{S}),(T, \mathcal{T})$ respectively. Let $\mathcal{F}$ be a sub-$\sigma$ algebra of $\mathcal{A}$ such that the conditional distribution $\mathbb{P}[\xi\in \cdot|\mathcal{F}]$ has a regular version $\mu:\Omega\times\mathcal{S}\to \R_{+}$ and $f$ be a measurable function on $S\times T$ with $\mathbb{E}[|f(\xi, \eta)|]<\infty$. If $\eta$ is $\mathcal{F}-$measurable, then 
    \[\mathbb{E}[f(\xi,\eta)|\mathcal{F}] = \displaystyle \int \mu(ds)f(s,\eta)\quad \mathbb{P}-a.s. \]
\end{theorem}

\begin{remark}
    To guarantee the existence of such regular version of conditional distributions, in general on needs to impose some regularity constraints on the space $S$, namely that it is Borel, see Theorem $1.8$ in \cite{kallenberg1997foundations}. Examples of Borel spaces are complete metric spaces, i.e. \textit{Polish} spaces. In particular all the function spaces considered with their respective uniform topologies on compacts sets are Polish, see \cite{kallenberg1997foundations}.
\end{remark}

We now prove a useful result involving absolute continuity, when independence is involved. In this particular case, computing the regular conditional distribution is rather simple, due to independence, though instructive in illustrating how the more general case works.\\

Let $(\Omega, \mathcal{F}, \mathbb{P})$ be a probability space and $I\subseteq \R$ be a closed interval. Suppose that $X$ is an $\mathcal{C}(I\times\R;\R)$-valued random variable and $Y$ is an $\mathcal{C}(\R;\R)$-valued random variable, with respect to the corresponding Borel sigma algebras over the topologies of local uniform convergence. Further, let $F:\mathcal{C}(I\times\R;\R)\times \mathcal{C}(\R;\R)\to \mathcal{C}(\R;\R)$ denote the functional
 \[
 (f,g)\mapsto \Bigl(y\mapsto \displaystyle\max_{x\in\R}(g(x)+f(y,x))\Bigr)
 \]
 Essentially the same argument as Lemma $4.1$ in \cite{schilling2021brownian} shows that the Borel sigma algebra on a space of functions with the usual sigma algebra mentioned above is the trace sigma algebra of the cylindrical sigma algebras on the respective spaces of functions. This implies that the functional $F$ is measurable, since it suffices to consider the measurability of the projections onto each coordinate in $I$. We can thus prove the following.\\

\begin{lemma}\label{lemma:cond independence}
   Suppose now that $X$ and $Y$ are independent. If the law $\xi_x$ of $F(x,Y)$ is absolutely continuous for all $x\in\R$ with respect to some finite measure $\mu$. Then, 
   \[
   \xi\lll\mu
   \]
   where $\xi$ is the law of $F(X,Y)$.
\end{lemma}
\begin{proof}
    Let $A$ be a Borel set in $\mathcal{C}(\R;\R)$ such that $\mu(A) = 0$. Then, we compute
    \[
    \xi(A) = \mathbb{P}(\{F(X,Y)\in A\}) = \mathbb{E}[\mathbf{1}_{F^{-1}(A)}(X,Y)]
    \]
    \[
    = \displaystyle\int \mathbf{1}_{F^{-1}(A)}(x,y)d\eta\otimes d\nu\,,
    \]
    where $\nu$ is the law of $Y$, $\eta$ is the law of $X$ and we used the change of variables theorem for measure-theoretic integrals and the independence of $X,Y$. By Fubini's theorem, we obtain 
    \[
    \xi(A) = \displaystyle\int \Bigl(\int \mathbf{1}_{F^{-1}(A)}(x,y)d\nu(y) \Bigr)d\eta(x)
    \]
    \[
    = \displaystyle\int \Bigl(\int \mathbf{1}_{F^{-1}(A)}(x,y)d\nu(y) \Bigr)d\eta(x) = \displaystyle\int \Bigl(\int \mathbf{1}_{F(x,\cdot)^{-1}(A)}(y)d\nu(y) \Bigr)d\eta(x)
    \]
    \[
    = \displaystyle\int \xi_x(A)d\eta(x) = 0
    \]
    by another application of change of variables. Hence, $\xi\lll\mu$.
\end{proof}

\chapter{Local Brownianness and regularity implications}
For convenience, we recall the definition of \textit{locally Brownian}. We say that random function $F$ in $\mathcal{C}(\R;\R)$ is called {\bf Brownian on compacts} if for all $y_1<y_2$, the law $\eta_{y_1, y_2}$ of the $\mathcal{C}_0[y_1,y_2]$-valued random variable
$$y\mapsto F(y)-F(y_1)$$
is absolutely continuous with respect to the law of a Brownian motion starting from $(y_1,0)$ with diffusion parameter $\sigma^2$ on $[y_1,y_2]$, call it $\mu_{y_1, y_2}$. Owing to convention, herein we use $\sigma^2=2$. We now prove the following.\\

In the paper of Quastel, Remenik and Matetski \cite{matetski2021kpz} that preceded the works \cite{DOV} and \cite{sarkar2021brownian}, another notion of local Brownianness is introduced, which we call \textbf{weak local Brownianness}. 

\begin{definition}{\bf (Weak Locally Brownian behavior)}\label{locbr}
The random function $F$ in $\mathcal{C}(\R;\R)$ is said to be \textbf{weakly locally Brownian} if 
\begin{equation}\label{eq: Brownian scaling}
F_\epsilon(t)=\epsilon^{-1}(F(\epsilon^2 t)-F(0))
\end{equation}
converges in law to a Brownian motion $B$ as $\epsilon\to 0$, in the uniform-on-compact topology.
\end{definition} 

\begin{remark}
In the work \cite{matetski2021kpz}, in Theorem $4.14$, they proved that the KPZ fixed point constructed as a scaling limit of the TASEP process (see section \ref{sec: non-gaussian}) is weakly locally Brownian.
\end{remark}

The corollary below clarifies the relationship between these notions of local Brownianness, particularly, that being Brownian on compacts implies weak locally Brownian behaviour. Hence, the former is a stronger notion than the latter. The proof proceeds by contradiction and essentially follows the argument in the paper by Dauvergne, Sarkar and Virag \cite{dauvergne2022three}.

\begin{corollary}
\label{c:wlocalbrown} Let the $\mathcal{C}(\R;\R)$-valued random element $F(\cdot)$ have law which is Brownian on compacts. Then $F(\cdot)$ is weakly locally Brownian.
\end{corollary}
\begin{proof} 
Assume for a contradiction the existence of an $F(\cdot)$ for which the above implication does not hold. Then, from the definition of weak convergence of measures, there exists a bounded continuous function $G:\mathcal{C}(\R;\R)\to \R$ and $\delta>0$ such that
\[
|\mathbb{E}(G(F_\epsilon))-\mathbb{E}(G(B))|>\delta
\]
along a subsequence of $\epsilon \to 0$. Now, by an approximation argument by simple functions, there exists a measurable set $A$ and a $\delta>0$ so that 
\[|\PP(F_\epsilon \in A)-\PP(B \in A)|>\delta\]
 along some subsequence of $\epsilon \to 0 $. First, assume that \[\PP(F_\epsilon\in A) < \PP(B\in A)-\delta\]
 along that subsequence. Furthermore, notice that since the law of the incremental segments $F_\epsilon(\cdot) - F_\epsilon(s)$ restricted to a time interval $[s,t]$ for some $0<s<t$ determine the law of any continuous process, we can choose the event $A$ so that it depends only on said increments. more explicitly, once convergence is established for such $A$, the argument for more general $A$ follows by a standard monotone class type argument since the above would imply convergence occurs for cylindrical sets too and the collection of sets for which convergence is established is readily seen to be a sigma algebra.\\

Now, pick a further subsequence $\epsilon_n$ so that $\epsilon_n/\epsilon_{n+1}>(t/s)^{1/2}$. Let $X_n=1(F_{\epsilon_n} \in A)$, and let $Y_n=(X_1+\cdots+X_n)/n$. Then
\[\mathbb{E} Y_n=\frac{1}{n}\sum_{i=1}^n\PP(F_{\epsilon_i}\in A)< \PP(B\in A)-\delta\]
for all $n$. Hence
\begin{equation}\label{eq:wtf}
\limsup_n \mathbb{E} Y_n\leq \PP(B\in A)-\delta\,.\
\end{equation}
  Now, if the law of $F$ is that of a Brownian motion, we claim that the $X_i$ are i.i.d. Indeed, to see independence, observe that the condition on the geometric decay of the subsequence $(\epsilon_n)$  implies $t\epsilon_{n+k}^2<t(s/t)^{k-1}\epsilon_n^2\leq s\epsilon_n^2$ for $k\geq 1, n \in \N$. Hence, $X_{k_1}X_{k_2}, \dots, X_{k_l}$ for $k_1<k_2<\dots<k_l$ are mutually independent as the respective events defining them depend on increments of $F$ evaluated on mutually disjoint intervals. The fact that they have the same law follows from the observation that the scaling performed in \ref{eq: Brownian scaling} preserves the law of Brownian motion. Thus, the law of large numbers implies
  \[
    Y_n \to \PP(B \in A),\quad n\to \infty \quad \text{ almost surely}\,.
  \]
   Since this convergence is almost sure, the same holds if the law of $F$ is absolutely continuous with respect to Brownian motion on $[0,1]$\footnote{for $\epsilon_1$ sufficiently small.}. The bounded convergence theorem then implies that
   \[\mathbb{E} Y_n \to \PP(B \in A)\,, \quad n\to \infty\]
   which contradicts \eqref{eq:wtf}. The case where along a subsequence
   \[\PP(F_\epsilon\in A) > \PP(B\in A)-\delta\]
  is dealt with entirely analogously, with essentially the same argument leading one to consider
  \[
  \PP(B\in A) < \liminf_n \mathbb{E} Y_n-\delta
  \]
  and proceed as before, which concludes the proof of the lemma.

\end{proof}
\begin{remark}\label{r:brhighdim} The same result can be slightly generalised, with the same proof, for two-sided Brownian motion scaled near zero, in arbitrary dimensions, that is, let $W$ have law which is absolutely continuous with respect to $B_d$, a $d$-dimensional Brownian motion, on compact sets for some $d\in \N$. Then
$$W_\epsilon(t)=\epsilon^{-1}(W(\epsilon^2t)-W(0))$$
converges in law to $B_d$ as $\epsilon\to 0$ in the uniform-on-compact topology.
\end{remark}

Finally, as a consequence of Theorem \ref{thm: Absolute continuity KPZ}, we also have that the KPZ fixed point enjoys H\"{o}lder $\frac{1}{2}-$ regularity, which is the content of the next corollary.

\begin{corollary}{(H\"{o}lder $\frac{1}{2}-$ regularity in space)}\label{cor: Holder reg}
Let $t>0$ and 
\[
h_t(y) = \sup_{x\in\R}(h_0(x)+\mathcal{S}(x,y)) \stackrel{d}{=} \sup_{x\in\R}(h_0(x)+\mathcal{L}(x,0;y,t))
\]
be the KPZ fixed point at time $t$, where $h_0(\cdot)$ is a $t-finitary$ initial condition. Then for any $\beta\in (0,\frac{1}{2})$ and $M<\infty$, 
\begin{equation}\label{eq: Holder reg}
\displaystyle\lim_{A\to\infty}\PP\left( \lVert h(t)\rVert_{\beta,\space [-M,M]}\geq A\right) = 0\,,
\end{equation}
where $\lVert \cdot \rVert_{\beta,\space [-M,M]}$ denotes the $\beta$-Holder norm on the interval $[-M,M]$.
\end{corollary}

\begin{proof}
    Fix $\beta\in (0,\frac{1}{2})$ and $M<\infty$. Observe that
    \[
    \lVert h(t)\rVert_{\beta,\space [-M,M]} \coloneqq 
    \sup_{\substack{x\neq y,\\x,y\in [-M,M]}}\frac{|h_t(x)-h_t(y)|}{|x-y|^\beta}
    \]
    \[
    = \sup_{\substack{x\neq y,\\x,y\in [-M,M]}}\frac{|(h_t(x)-h_t(-M))-(h_t(y)-h_t(-M))|}{|x-y|^\beta} = \lVert h_t(\cdot)-h_t(-M)\rVert_{\beta,\space [-M,M]}\,.
    \]
    Since $h_t$ is Brownian on compacts, it follows that the law of $h_y(\cdot)-h_t(-M)$ restricted to $[-M,M]$ is absolutely continuous with respect to the law of a Brownian motion $B$, $\mu$ starting from $(-M,0)$ with diffusion parameter $2$ on $[-M,M]$. Note that after possibly augmenting the probability space, we can take $B$ to live in the same probability space as $h_t(\cdot)$.\\

    Now, Corollary $1.20$ in \cite{morters2010brownian} implies that $B$ is almost surely locally H\"{o}lder $\beta$ continuous. More precisely,
    almost surely, for all $x\in[-M,M]$, there exist random $\epsilon_x, C_x >0$ such that
    \[
    |B(y)-B(x)|\leq C_x |x-y|^\beta, \quad \text{for all } y\in[-M,M] \text{ with } |x-y|\leq \epsilon_x\,.
    \]
    \noindent Covering $[-M, M]$ by intervals $[x-\epsilon_x, x+\epsilon_x]$ invoking compactness and passing to a finite subcover yields that
    $\lVert B(\cdot)\rVert_{\beta,\space [-M,M]} < \infty$ almost surely. This in conjunction with the aforementioned absolute continuity yields that
    \[
    \lVert h_t(\cdot)-h_t(-M)\rVert_{\beta,\space [-M,M]}<\infty \quad \text{ almost surely}\,,
    \]
    which implies \ref{eq: Holder reg} by dominated convergence, as required.
    
\end{proof}

\begin{remark}
    Note that this spatial regularity result obtained for the KPZ fixed point was also shown to hold in \cite{matetski2021kpz}, using exact formulas. However, this relatively simple deduction using standard properties of Brownian motion shows the strength of the improved regularity which is enjoyed by the KPZ fixed point, namely being Brownian on compacts.
\end{remark}

\bibliographystyle{plain} 
\bibliography{refs} 
\end{document}